\documentclass[11pt,oneside]{amsart}
\usepackage{pinlabel}
\usepackage{amsmath, amsthm, amssymb, amscd, mathrsfs, eucal}
\usepackage{enumerate}
\usepackage[T1]{fontenc}
\usepackage{hyperref}
\usepackage{url}
\usepackage{float}
\usepackage{mathtools}
\usepackage{color}
\usepackage{geometry}
\usepackage{graphicx}

\newtheorem{theorem}{Theorem}[section]
\newtheorem{lemma}[theorem]{Lemma}
\newtheorem{corollary}[theorem]{Corollary}

\newtheorem{proposition}[theorem]{Proposition}
\newtheorem{question}[theorem]{Question}

\newtheorem*{claim*}{Claim}

\theoremstyle{definition}
\newtheorem{example}[theorem]{Example}
\newtheorem{definition}[theorem]{Definition}
\newtheorem{construction}[theorem]{Construction}

\theoremstyle{theorem}
\newtheorem{theoremA}{Theorem}

\theoremstyle{remark}
\newtheorem{remark}[theorem]{Remark}

\usepackage{color}
\usepackage{pdfcolmk}

\newcommand{\rom}[1]{\text{\uppercase\expandafter{\romannumeral #1\relax}}}

\newcounter{lcomments}

\newcounter{ycomments}

\newcounter{dcomments}

\newcommand{\inj}{\hookrightarrow}

\def\Z{\mathbb Z}
\def\N{\mathbb N}
\def\R{\mathbb R}

\def\Q{\mathbb Q}

\def\defeq{\vcentcolon=}

\def\sc{\mathrm{sc}}
\def\lamhat{\hat{\lambda}}
\def\sign{\mathrm{sign}}

\makeatletter
\newsavebox{\@brx}
\newcommand{\llangle}[1][]{\savebox{\@brx}{\(\m@th{#1\langle}\)}%
	\mathopen{\copy\@brx\kern-0.5\wd\@brx\usebox{\@brx}}}
\newcommand{\rrangle}[1][]{\savebox{\@brx}{\(\m@th{#1\rangle}\)}%
	\mathclose{\copy\@brx\kern-0.5\wd\@brx\usebox{\@brx}}}
\makeatother

\numberwithin{equation}{section}

\usepackage{pdfpages}

\title{The Wiegold problem and free products of left-orderable groups}
\author{Lvzhou Chen}
\address{Department of Mathematics\\ Purdue University\\ West Lafayette, Indiana, USA}
\email[L.~Chen]{lvzhou@purdue.edu}

\author{Yash Lodha}
\address{Department of Mathematics\\ Purdue University\\ West Lafayette, Indiana, USA}
\email[Y.~Lodha]{ylodha@purdue.edu}

\begin{document}

\begin{abstract}
    A group has normal rank (or weight) greater than one if no single element normally generates the group.
    The Wiegold problem from $1976$ asks about the existence of a finitely generated perfect group of normal rank greater than one. 
	We show that any free product of nontrivial left-orderable groups has normal rank greater than one.
	This solves the Wiegold problem by taking free products of finitely generated perfect left-orderable groups, a plethora of which are known to exist.
	We obtain our estimate of normal rank by a topological argument, proving a type of spectral gap property for an unsigned version of stable commutator length. A key ingredient in the proof is an intricate new construction of a family of left-orders on free products of two left-orderable groups. 
\end{abstract}

%

\maketitle

\section{Introduction}\label{sec: intro}
A fundamental concept in group theory is 
the \emph{normal rank} (or \emph{weight}) of a group, which is the smallest cardinality of a set of elements that normally generates the group. 
One obtains a lower bound by counting the minimal number of factors in the abelianization, viewed as a direct sum.
Providing other estimates from below has proved to be notoriously difficult.
Such estimates are also of interest in $3$-manifold topology, as the normal rank of $\pi_1(M)$ is a lower bound of the so-called Dehn surgery number of $M^3$; see Section~\ref{subsec: connections}.

A longstanding open problem posed by Wiegold in $1976$ (Problem FP14 in \cite{MR1921705}, Problem $5.52$ in \cite{Kourovka}), asks if there exist finitely generated perfect groups with normal rank greater than one, for which we give a positive answer:
\begin{theoremA}\label{thmA: Wiegold}
    There exist finitely generated (even finitely presented) perfect groups that have normal rank greater than $1$. 
\end{theoremA}

There has been remarkably little progress on the problem since it was posed, besides Wiegold's elementary proof that finite perfect groups have normal rank $1$. There are many candidates of such examples, such as free products of finitely generated torsion-free perfect groups. However, proving the lower bound on the normal rank remained intractable.

We prove the desired bound for free products of \emph{left-orderable} groups, and thus provide a family of examples for Theorem~\ref{thmA: Wiegold}.
A group $G$ is left-orderable if it admits a total order $<$ such that $f<g$ implies that $hf<hg$ for all $f,g,h\in G$.
For countable groups, this is equivalent to admitting an embedding in $\textup{Homeo}^+(\R)$ \cite{GOD}.
The landscape of left-orderable groups is vast and includes free and surface groups (more generally, locally indicable groups \cite{BurnsHale}), braid groups, and virtually the fundamental groups of closed hyperbolic $3$-manifolds \cite{Agol}. 
Finitely generated perfect left-orderable groups are plentiful with examples going back to the $1980$s \cite{GhysSergiescu}. 
See \cite{GOD} for a comprehensive reference on left-orderable groups.

Here is the general result we prove, where $\llangle w \rrangle$ is the smallest normal subgroup of $G$ containing $w$.
\begin{theoremA}\label{thmA: main}
    Any free product $G=A\star B$ of nontrivial left-orderable groups $A$ and $B$ has normal rank greater than $1$. More precisely, for any $w\in G$ not conjugate into $A$, the natural map $A\inj A\star B$ induces an \emph{injection} $A\inj (A\star B)/\llangle w \rrangle$.
\end{theoremA}

We remark that the same estimate of normal rank holds for free products of groups, where each free factor admits a nontrivial left-orderable quotient.

Theorem~\ref{thmA: Wiegold} is a corollary of Theorem~\ref{thmA: main} by taking $A$ and $B$ to be finitely generated perfect left-orderable groups. Some prominent examples of finitely presentable perfect left-orderable groups include: the commutator subgroup of the braid group $B_n$ for $n\geq 5$ \cite{Dehornoy,GOD}, a central extension of Thompson's group $T$ \cite{GhysSergiescu} and a central extension of the $(2,3,7)$-triangle group \cite{GOD}, and Higman's group \cite{RivasTriestino}. 
There are also many $\Z$-homology $3$-spheres with left-orderable fundamental groups \cite{LSpace,CullerDunfield}.
The second author with Hyde constructed the first examples of finitely generated simple left-orderable groups \cite{HydeLodha}, and even finitely presentable ones \cite{HydeLodhaFP}. The former family contains continuum many groups (up to isomorphism).

One can also build other examples for the Wiegold problem based on the ones from Theorem~\ref{thmA: main}, so that the group has other properties: freely indecomposable, linear, residually finite, and hyperbolic; see Theorem~\ref{thmA: Dehn surgery}.

\subsection{Connections to other problems and results}\label{subsec: connections}
We now explain how our results connect to other problems, and compare them with previously known results. The method of proof for Theorem~\ref{thmA: main} will be explained in Section~\ref{sec: structure of proof}.

Most known approaches towards the Wiegold problem involve showing that a free product $G=\star_{i\in I} A_i$ of nontrivial groups $A_i$ has normal rank greater than $1$, under suitable assumptions.

For $|I|\ge3$, an unsolved conjecture of Cameron Gordon \cite[Conjecture 9.5]{Gordon:DehnSurg} predicts that such a free product always has normal rank greater than $1$. When the factors are finite cyclic groups, this is the Scott--Wiegold conjecture, confirmed by James Howie \cite{Howie_cyclic}. It seems tempting to make this work for a free product of three (even finite) perfect groups, but no one has succeeded.

When $|I|=2$, it is well-known that torsion can enable a free product $G=A\star B$ to have normal rank $1$. For example, $G/\llangle w\rrangle$ is trivial for $w=ab$ in $G=(\Z/2)\star(\Z/3)$, where $a$ and $b$ are generators of $\Z/2$ and $\Z/3$ respectively. The same argument works when $A,B$ are simple groups containing torsion elements of coprime orders. So some assumptions should be imposed on $A$ and $B$ for $G$ to have normal rank greater than $1$.

Since the early $1980$s, Theorem~\ref{thmA: main} was known to be true under the much stronger assumption that each free factor is \emph{locally indicable}, meaning that every nontrivial finitely generated subgroup surjects onto $\Z$; see the independent work of Brodski\u{\i} \cite{Brodskii}, Howie \cite{Howie_LocIndFrei} and Short \cite{Short}. However, finitely generated locally indicable groups can never be perfect.
Theorem~\ref{thmA: main} is also suspected to hold under the weaker assumption that each free factor is torsion-free; see Kirby's ($1970$s) problem list \cite[Problem 66]{Kirby:oldlist} contributed by Freedman. 
Such generalizations are challenging due to the lack of structures.
Local indicability implies left-orderability, which in turn implies torsion-freeness.

If one takes $B=\Z$ in Theorem~\ref{thmA: main}, it becomes easier to show that $G$ has normal rank greater than one, but there is a surjection $p:G=A\star\Z\to \Z$, and so $G$ is not perfect.
The Howie (or Kervaire--Laudenbach) conjecture \cite{Howie_LocIndFrei} asserts that Theorem~\ref{thmA: main} holds for $G=A\star\Z$ with any nontrivial group $A$ whenever $p(w)\neq0$. This has been confirmed when $A$ is residually finite (by Gerstenhaber--Rothaus \cite{GerstenhaberRothaus}), hyperlinear (by Pestov \cite{Pestov}), or torsion-free\footnote{under the stronger assumption $p(w)=\pm 1$} (by Klyachko \cite{Klyachko} and a recent new proof by the first author \cite{Chen:Kervaire}).
Actually, if $A$ is torsion-free, the Levin conjecture \cite{Levin} asserts that the result holds with the weaker assumption that $w$ is not conjugate into $A$. Theorem~\ref{thmA: main} confirms this for $A$ left-orderable since $B=\Z$.

Considerable interest in the notion of normal rank comes from $3$-manifold topology. The Lickorish--Wallace theorem \cite{Wallace,Lickorish} shows that any closed, orientable, connected $3$-manifold $M$ is the result of a Dehn surgery on some $n$-component link $L$ in $S^3$, and the minimal number $n$, called the \emph{Dehn surgery number}, has been considered as a complexity of $M$. There are many articles studying this \cite{surg1,surg2,HomLidman,surg4}, but it is very difficult to give lower bounds. In fact, no example of Dehn surgery number greater than two is known\footnote{Ali Daemi and Mike Miller Eismeier claimed the existence of such examples using gauge theory methods, but their work has not appeared yet.}; see the recent work \cite{LiuPiccirillo} for a more detailed summary of the relevant work.
The normal rank of $\pi_1(M)$ is a natural lower bound of the Dehn surgery number, since the fundamental group of $S^3\setminus L$ has normal rank at most $n$, and so is the fundamental group of $M$ if it is obtained from a Dehn surgery on $S^3\setminus L$ since $\pi_1(M)$ is a quotient of $\pi_1(S^3\setminus L)$. Actually, a finitely generated group has normal rank $1$ if and only if it is the quotient of some knot group \cite{Gonzalez,Johnson}.

In this view, the Gordon conjecture \cite[Conjecture 9.5]{Gordon:DehnSurg} is a generalization of the three-summand conjecture: A connected sum of three $3$-manifolds not homeomorphic to $S^3$ must have Dehn surgery number greater than $1$. This is in turn related to the cabling conjecture \cite{AcunaShort}.
Left-orderability of fundamental groups of $3$-manifolds is also well-studied in the context of the L-space conjecture \cite{LSpace}.


Due to this connection to the Dehn surgery number, we have the following corollary of Theorem~\ref{thmA: main}.
\begin{corollary}\label{cor: connected sum}
    If $M=M_1\# M_2$ is a connected sum of closed orientable connected $3$-manifolds with $\pi_1(M_i)$ nontrivial and left-orderable for $i=1,2$, then $\pi_1(M)$ has normal rank at least $2$, in particular, $M$ cannot be the result of a Dehn surgery on a knot in $S^3$.
\end{corollary}

The second assertion is a consequence of a theorem of Gordon--Luecke \cite[Theorem 3]{GordonLuecke}: If a reducible $3$-manifold is the result of a Dehn surgery on a knot in $S^3$, then one summand is a lens space.

Moreover, as Nathan Dunfield pointed out to us, one can find irreducible manifolds that have degree-one maps to examples from Corollary~\ref{cor: connected sum}. Combining this with a trick we learned from Tye Lidman, we obtain Theorem~\ref{thmA: Dehn surgery} below.

\begin{theoremA}\label{thmA: Dehn surgery}
    There are infinitely many different hyperbolic integer homology $3$-spheres $M$ such that $\pi_1(M)$ has normal rank at least $2$, and $M$ cannot be obtained from a Dehn surgery on a knot in $S^3$, $S^2\times S^1$, or any lens space.
\end{theoremA}

First, this gives infinitely many solutions to the Wiegold problem that are linear groups. Second, it strengthens theorems of Auckly \cite{Auckly} and Hom--Lidman \cite{HomLidman} about the existence of (infinitely many) hyperbolic integer homology $3$-spheres with Dehn surgery number at least $2$ (\cite[Problem 3.6(C)]{Kirby95}). Our examples have normal rank at least $2$ and our proof does not use any gauge theory. Moreover, the construction is very flexible; see Construction~\ref{construction}. 
See Section~\ref{sec: Dehn surgery} for more details.




It is natural to ask if the following generalization of Theorem~\ref{thmA: main} holds.
\begin{question}
    Let $G$ be a free product of $n$ nontrivial left-orderable groups. Is the normal rank of $G$ at least $n$?
\end{question}
A positive answer would imply that the connected sum of $n$ $3$-manifolds with nontrivial left-orderable fundamental groups must have Dehn surgery number at least $n$.

There have been some attempts to approach the Wiegold problem from other directions, which still remain mysterious.
In the article \cite{OsinThom}, Osin and Thom provide a conjectural connection between the notion of normal rank of a countable torsion-free group and its first $l^2$-betti number, motivated by the Wiegold problem. In \cite{MonodOzawaThom}, Monod, Ozawa and Thom formulate a version of the Wiegold problem for \emph{irng}'s which are rings that are possibly non-unital, demonstrating that a solution in this setting would lead to a solution of the original Wiegold problem. Both approaches remain unsolved. 
Also in a contrasting result \cite{MonodEisenmann}, Monod and Eisenmann proved that a compactly generated perfect locally compact group without infinite discrete quotients is normally generated by a single element. This emphasizes the difference with the topological group setting.

\section{The structure of the proof}\label{sec: structure of proof}

We sketch the proof of Theorem \ref{thmA: main} and introduce the two main ingredients. 

The notion of \emph{right-orderability} is defined in the same way as left-orderability, except that the total order is required to be invariant under right multiplication instead. A group admits a left-order if and only if it admits a right-order. In this article, all actions will be right actions, so we shall use the notion of right-orderability in the rest of this paper.

Let $A,B$ be right-orderable groups and $G=A\star B$.
Up to conjugation, we express $w$ as a cyclically reduced word:
$$w=a_1b_1\cdots a_n b_n\in G\qquad a_i\in A\setminus \{id\}, b_i\in B\setminus \{id\}.$$

Suppose some $a\in A\setminus \{id\}$ lies in the subgroup $\llangle w\rrangle$ normally generated by $w$, we get an equation of the form: 
$$a= (g_1 w^{n_1} g_1^{-1})(g_2 w^{n_2} g_2^{-1})\cdots (g_k w^{n_k} g_k^{-1}),$$
where $k\in \Z_+$, each $g_i\in G$ and $n_i\in\Z\setminus\{0\}$ for all $1\le i\le k$.

This gives rise to a surface map $f:S\to X$, where $X$ is a $K(G,1)$, $S$ is a sphere with $k+1$ boundary components, and $f$ takes the boundary components to loops in $X$ representing the conjugacy classes of $a, w^{n_1},\cdots, w^{n_k}$ respectively; see Figure~\ref{fig: fromeqn}.
We aim to show that such a surface map is too simple to exist in the sense that the complexity of the surface, measured by $-\chi(S)=k-1$, is strictly less than the complexity of the boundary, measured by $\sum_i |n_i|$ (which is at least $k$).


We show that the inequality should go the other way in Theorem~\ref{thmA: spectral gap theorem} below, for a more general class of surfaces (allowing genus and so on), which we call \emph{$w$-admissible surfaces}; see Definition \ref{def: admissible}. 
For such a surface $S$, we measure the complexity of the boundary by a notion of degree $\deg(S)$, which is an \emph{unsigned} count of the number of copies of $w^{\pm1}$ on the boundary. We ask the surface to be boundary-incompressible (Definition \ref{def: boundary-incompressible}) to ensure a genuine count of $\deg(S)$.
The surface $S$ above arising from equations fits into the definition, 
where $\deg(S)=\sum_i |n_i|$, and it is boundary-incompressible if we take an equation of minimal length $k$; see Example \ref{example: equations in groups}. Thus the inequality we show in Theorem~\ref{thmA: spectral gap theorem} below gives a contradiction.

\begin{theoremA}\label{thmA: spectral gap theorem}
    Let $G=A\star B$ be a free product of right-orderable groups $A$ and $B$, and let $w$ be an element of $G$ not conjugate into $A$ or $B$. Then for any boundary-incompressible $w$-admissible surface $S$ without $2$-sphere or disk components, we have
    $$-\chi(S)\ge \deg(S).$$
\end{theoremA}

If we express $w$ as a cyclically reduced word $a_1b_1\cdots a_nb_n$, Theorem~\ref{thmA: spectral gap theorem} holds under the weaker assumption that the finitely generated subgroups $A'=\langle a_1,\ldots,a_n\rangle$ and $B'=\langle b_1,\ldots,b_n\rangle$ are right-orderable. Similarly for Theorem~\ref{thmA: spectral gap from stacking} below; see Theorem~\ref{thm: spectral gap from stacking}.

A similar result in the context of HNN extensions was recently shown by the first author \cite[Theorem A]{Chen:Kervaire} to give a new proof of the Klyachko theorem \cite{Klyachko}. These are inspired by the spectral gap property in the context of stable commutator length \cite{DH91,Chen:sclfpgap,CH:sclgap,Heuer}, where the degree is the \emph{signed} count $|\sum_i n_i|$. In that context, the analog of Theorem~\ref{thmA: spectral gap theorem} holds under the weaker torsion-freeness assumption on the free factors \cite{Chen:sclfpgap}. Here we are interested in the stronger inequalities involving the \emph{unsigned} degree, for which we impose stronger assumptions. Similar unsigned or unoriented variants of stable commutator length has been studied by Duncan--Howie, Larsen Louder, Doron Puder, and Henry Wilton \cite{DH91,LouderWilton:coherence,Wilton:surfsubgroup,Wilton:curvinv,Puder}, which have been useful in understanding surface subgroups, one-relator groups, and more.

To prove Theorem~\ref{thmA: spectral gap theorem}, we have two main ingredients, one topological and the other dynamical.
First, we use a topological argument that gives the desired estimate of the Euler characteristic (Theorem~\ref{thmA: spectral gap from stacking} below), assuming the existence of a combinatorial labeling. Such a labeling can be derived from a dynamical condition that we call \emph{relative stacking} (Definition~\ref{def: relative stacking}), generalizing the notion of stacking from \cite{LouderWilton:stacking}.

\begin{theoremA}\label{thmA: spectral gap from stacking}
    For any free product $G=A\star B$ and $w\in A\star B$ expressed as a cyclically reduced word. If there is a relative stacking of $w$, then for any boundary-incompressible $w$-admissible surface $S$ without sphere or disk components, we have
    $$-\chi(S)\ge \deg(S).$$
\end{theoremA}

Here is the key notion of relative stacking.
\begin{definition}[Relative stacking]\label{def: relative stacking}
    Consider an action $\sigma:G=A\star B\to\textup{Homeo}^+(\R)$, not necessarily faithful.
    For a cyclically reduced word $w=a_1b_1\cdots a_n b_n$ with $a_i\in A\setminus\{id\}, b_i\in B\setminus \{id\}$,
    the \emph{trajectory} of some $x\in\R$ under $w$ is the multiset of points:
    $$\Omega(w,x)=\{x\cdot \sigma(a_1b_1\ldots a_ib_i)\mid 1\leq i\leq n\}\cup \{x\cdot \sigma(a_1b_1\ldots a_i)\mid 1\leq i\leq n\},$$
    which are the images of $x$ under the action of the nonempty prefixes of $w$.

    We say the trajectory of $x$ under $w$ is \emph{stable} if:
    \begin{enumerate}
        \item Each element in $\Omega(w,x)$ occurs exactly once.
        \item $x\cdot \sigma(w)=x$.
    \end{enumerate}

    A \emph{relative stacking} of $w$ is a right action $\sigma$ of $A\star B$ on $\R$ by orientation-preserving homeomorphisms of the real line and a point $x\in \R$, such that $\Omega(w,x)$ is stable.
\end{definition}

The second main input is to show the existence of relative stacking. It only works when $w$ is not a proper power, but Theorem~\ref{thmA: spectral gap theorem} is easier when $w$ is a proper power.
In fact, for the purposes of our main Theorem~\ref{thmA: main}, we only need to consider $w$ which are not proper powers since $\llangle w^n\rrangle\subset \llangle w\rrangle$.

\begin{theoremA}\label{thmA: Stacking}
    Let $G=A\star B$ be a free product of countable right-orderable groups $A$ and $B$. For any $w\in A\star B$ cyclically reduced of length at least $2$, a relative stacking of $w$ exists if and only if $w$ is not a proper power.
\end{theoremA}

To prove Theorem~\ref{thmA: Stacking}, we provide a method to solve certain equations and inequations over the group where the variables are a $G$-action on $\R$ and a chosen point in $\R$. This is done using systematic blow-ups of fixed group actions on $\R$, reducing it to the following problem: Given any nonempty proper prefix $u$ of $w$, find an action $\sigma:A\star B\to \textup{Homeo}^+(\R)$ and $x\in \R$ that satisfies $x\cdot\sigma(w)=x$ and $x\cdot\sigma(u)\neq x$. For this, we build an action $\sigma$ with a closed interval $I$ such that:
\begin{enumerate}
    \item $I\cdot \sigma(w)\subseteq I$.
    \item $I\cdot \sigma(u)\cap I=\emptyset$.
\end{enumerate}
Such an action is built by developing a method that we call \emph{dynamical arrangements}, which are careful combinatorial encodings of actions of $A\star B$ on $\R$.
This finishes the proof since the first item ensures the existence of a point $x\in I$ fixed by $\sigma(w)$ using the intermediate value theorem, and the second item ensures the $\sigma(u)$ has no fixed point in $I$.

\subsection{Organization of the paper}\label{subsec: Organization}
We give the topological backgrounds on $w$-admissible surfaces in Section~\ref{sec: adm} and then prove Theorem~\ref{thmA: spectral gap from stacking} in Section~\ref{sec: topol argument}. Then we show the existence of relative stacking in Section~\ref{sec: rel stacking}. Finally we give the formal proofs of main results in Section~\ref{sec: proof of main results} and obtain the applications to Dehn surgeries in Section~\ref{sec: Dehn surgery}.

\section{Admissible surfaces}\label{sec: adm}
In this section, we define the necessary terminology about admissible surfaces and introduce the basic facts. Most of this is parallel to what is developed in the previous work of the first author \cite{Chen:Kervaire}, but here we focus on free products instead of HNN extensions. We include all details for completeness.

Given a group $G$ and a collection of proper subgroups $\{A_i\}_{i\in I}$, where $I$ could be empty, let $X$ be a connected topological space with $\pi_1(X)=G$. The only case of interest in this paper is when $G=\star_{i\in I} A_i$ with $|I|=2$.

\begin{definition}[$w$-admissible]\label{def: admissible}
    Given $w\in G$, a map $f:S\to X$ from a compact oriented surface $S$ is called a \emph{$w$-admissible surface} in $G$ (or $X$) relative to $\{A_i\}_{i\in I}$ if: 
    \begin{enumerate}

    \item The image of each boundary component of $S$
    \begin{itemize}
        \item either represents the conjugacy class of $w^n$ for some $n\in\Z\setminus\{0\}$ (we refer to the union of components of this kind as the \emph{$w$-boundary}),
        \item or represents a conjugacy class in $G$ that intersects some $A_i$ nontrivially (we refer to a boundary component of the this kind as an \emph{$A_i$-boundary}).
    \end{itemize}    
        \item We require that the $w$-boundary is non-empty.
    \end{enumerate}
     
    The $A_i$-boundary can be empty, but it is important to allow its existence (see Example \ref{example: equations in groups}). See the left of Figure \ref{fig: adm_and_bdrycompressible} for an illustration.

    We denote a (relative) $w$-admissible surface by the pair $(f,S)$,
    but we will often simply denote it as $S$ unless we would like to emphasize the map $f$.
    
    For each $w$-boundary component representing $w^n$, we define its \emph{degree} to be $|n|$. Then we define the \emph{degree} of $S$ as the sum of degrees over all $w$-boundary components and denote it as $\deg(S)$. For instance, the surface $S$ in Figure \ref{fig: adm_and_bdrycompressible} has degree $k+m+n$.
\end{definition}
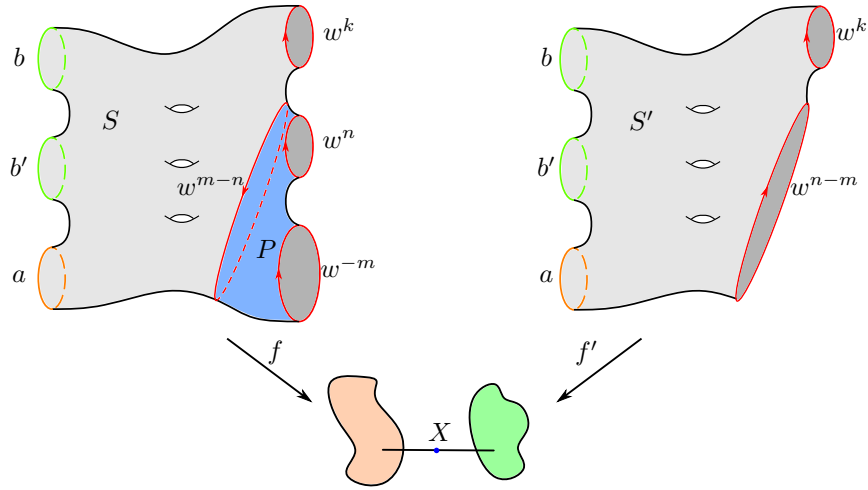
\begin{figure}
	\labellist
	\small \hair 2pt
	\pinlabel $b$ at -10 235
	\pinlabel $b'$ at -10 175
    \pinlabel $a$ at -10 115
	\pinlabel $S$ at 40 200
	\pinlabel $X$ at 220 30
	\pinlabel $P$ at 125 130
	\pinlabel $f$ at 130 73
	\pinlabel $f'$ at 300 73
	\pinlabel $w^{-m}$ at 170 120
	\pinlabel $w^{m-n}$ at 95 165
	\pinlabel $w^n$ at 165 190
	\pinlabel $w^k$ at 165 250
	
	\pinlabel $b$ at 278 235
	\pinlabel $b'$ at 278 175
    \pinlabel $a$ at 278 115
	\pinlabel $S'$ at 330 200
	\pinlabel $w^{n-m}$ at 430 165
	\pinlabel $w^k$ at 445 250
	\endlabellist
	\centering
	\includegraphics[scale=0.7]{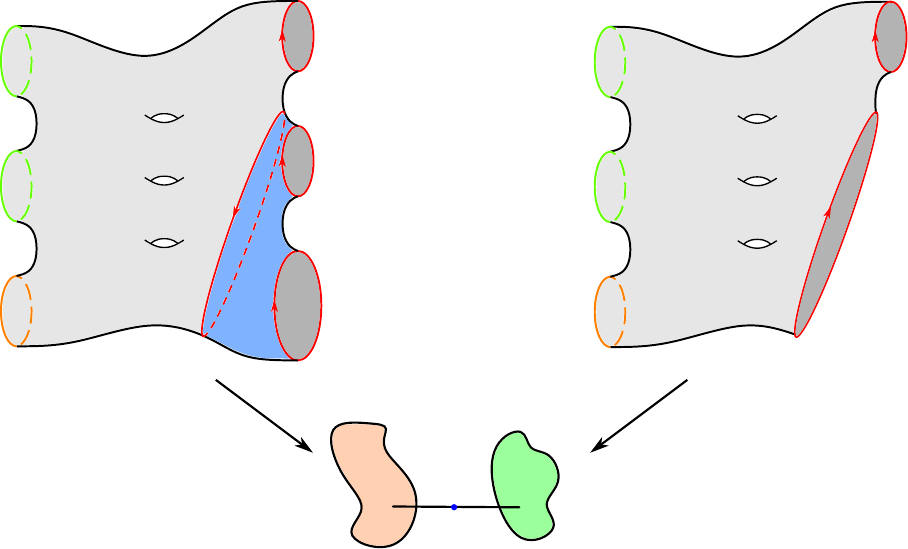}
	\caption{On the left is a $w$-admissible surface $(f,S)$ into a space $X$ with $\pi_1(X)=A\star B$, where $S$ has two boundary components representing $b,b'\in B$, one boundary component representing $a\in A$, and three $w$-boundary components representing $w^k,w^n,w^{-m}$ for some $k,m,n\in\Z_+$. The subsurface $P$ witnesses its boundary-compressibility, and $S'=S\setminus P$ on the right is the simplified $w$-admissible surface, whose boundary representing $w^{n-m}$ (with the orientation induced from $S'$) needs to be further capped off by a disk if $m=n$.}\label{fig: adm_and_bdrycompressible}
\end{figure}

\begin{remark}
    This is analogous to the definition of admissible surfaces relative to a collection of subgroups in the context of relative stable commutator length; see \cite[Definition 2.8]{Chen:sclBS}, and also \cite[Notation 2.5]{Cal:rational} for the absolute version. However, in that context, the degree is counted \emph{with signs}, namely, a boundary component representing $w^n$ with $n<0$ contributes $-n$ to the degree. The unsigned version of degree here is an upper bound of the absolute value of the signed version.
\end{remark}

To avoid a dummy count of degree, we will often require $w$-admissible surfaces to be \emph{boundary-incompressible}, following \cite[Definition 2.4]{Chen:Kervaire}.
\begin{definition}[boundary-incompressibility]\label{def: boundary-incompressible}
    A $w$-admissible surface $(f,S)$ relative to $\{A_i\}_{i\in I}$ is \emph{boundary-compressible} if there is an embedded subsurface $P\subset S$ homeomorphic to a pair of pants, such that two boundary components of $P$ are on the $w$-boundary of $S$ and represent the conjugacy class of $w^n$ and $w^{-m}$ for some $m,n\in\Z_+$, and the third boundary component of $P$ is in the interior of $S$ and, with the induced orientation from $P$, represents the conjugacy class of $w^{m-n}$ under the map $f$; see the left of Figure \ref{fig: adm_and_bdrycompressible}.

    We say $(f,S)$ is \emph{boundary-incompressible} if there is no such $P$.
\end{definition}

\begin{remark}\label{rmk:boundary-compression}
    One can compress a boundary-compressible $w$-admissible surface $S$ as follows to get a simpler surface that is either $w$-admissible again or has no $w$-boundary left; see Figure \ref{fig: adm_and_bdrycompressible} for an illustration. Let $S'$ be the closure of $S\setminus P$ equipped with the restriction of $f$. This ``merges'' the two $w$-boundary components of $S$ representing $w^n$ and $w^{-m}$ into one that represents the conjugacy class of $w^{n-m}$ (with the induced orientation from $S'$). If $m\neq n$, then $(f,S')$ is a $w$-admissible surface already. 
    If $m=n$, then the new boundary is null-homotopic in $X$ and thus bounds a disk $g:D\to X$. We use this to cap off $S'$ to get a surface $S''\defeq S'\cup D$ with a well-defined map $f''$ using $f$ and $g$, so that $(f'',S'')$ is a $w$-admissible surface except that possibly there is no $w$-boundary left. The surfaces $S'$ and $S''$ are simpler in the sense that $-\chi(S')=-\chi(S)-1<-\chi(S)$ and $-\chi(S'')=-\chi(S')-1<-\chi(S)$.
\end{remark}

The pair of pants $P$ in Definition~\ref{def: boundary-incompressible} shrinks down to a neighborhood of the union of the two $w$-boundary components $C,C'$ with a proper arc $\gamma\subset S$ connecting them. One can characterize boundary-incompressibility in terms of arcs (like $\gamma$); compare to the second part of \cite[Definition 3.1]{Puder}.

Here is the key example relating $w$-admissible surfaces to equations in a group.
\begin{example}\label{example: equations in groups}
    Let $a\neq id\in A$, which we treat as an element in $G=A\star B$. Let $w\in G$ be an element not conjugate into $A$ or $B$. Let $\llangle w\rrangle$ be the subgroup normally generated by $w$, then $a\in \llangle w \rrangle$ if and only if one can express $a$ in an equation of the form below:
    \begin{equation}\label{eqn: group equation}
        a= (g_1 w^{n_1} g_1^{-1})(g_2 w^{n_2} g_2^{-1})\cdots (g_k w^{n_k} g_k^{-1}),
    \end{equation}
    where $k\in \Z_+$, each $g_i\in G$ and $n_i\neq0\in\Z$ for all $1\le i\le k$.
    
    Each such equation corresponds to a surface map $f:S\to X$, where $S$ is a sphere with $k+1$ boundary components, where one boundary component represents the conjugacy class of $a$, and the remaining $k$ components represent $w^{n_i}$ for $1\le i\le k$ respectively; see Figure~\ref{fig: fromeqn}. This makes $(f,S)$ a $w$-admissible surface relative to $\{A,B\}$ (or just $\{A\}$), whose degree is $\sum_i |n_i|$.

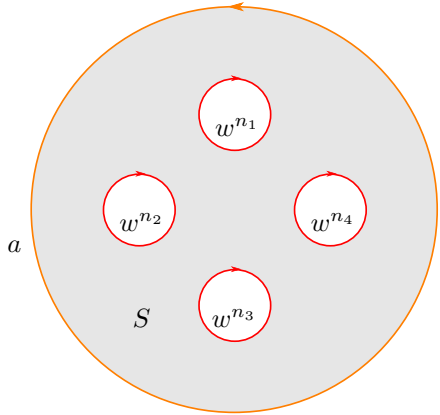
\begin{figure}
	\labellist
	\small \hair 2pt
	\pinlabel $a$ at -10 105
	\pinlabel $S$ at 70 60
	\pinlabel $w^{n_1}$ at 130 180
	\pinlabel $w^{n_2}$ at 70 120
	\pinlabel $w^{n_3}$ at 128 60
	\pinlabel $w^{n_4}$ at 190 120
	\endlabellist
	\centering
	\includegraphics[scale=0.6]{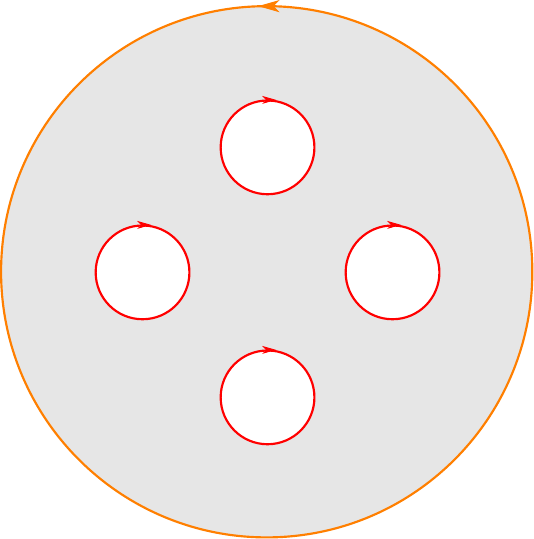}
	\caption{A $w$-admissible surface $S$ corresponding to an equation of the form (\ref{eqn: group equation}) with $k=4$.}\label{fig: fromeqn}
\end{figure}

    If $S$ is boundary-compressible, then the compression as in Remark \ref{rmk:boundary-compression} yields a $w$-admissible surface $S'$ with fewer $w$-boundary components, which corresponds to an equation with strictly smaller $k$. Since $a$ is nontrivial and the corresponding component in $S$ does not change in the process, we cannot run out of $w$-boundary components. It follows that, if we consider an equation for $a$ of the form (\ref{eqn: group equation}) with minimal $k$, then the corresponding $w$-admissible surface $S$ must be boundary-incompressible. These are the surfaces that we will apply our spectral gap Theorem \ref{thmA: spectral gap theorem} to.
\end{example}

We measure the complexity of a ($w$-admissible) surface $S$ by the negative modified Euler characteristic $-\chi^-(S)$ defined as follows: For each connected component $\Sigma$ of $S$, define $\chi^-(\Sigma)\defeq\min(\chi(\Sigma),0)$; then define $\chi^-(S)$ as the sum of $\chi^-(\Sigma)$ over all components of $S$. Equivalently, $\chi^-(S)$ is the Euler characteristic of $S$ after removing all disk or sphere components. For our purposes, we can restrict our attention to $w$-admissible surfaces where each connected component contains some $w$-boundary component, so there is no disk (if $w$ has infinite order) or sphere component, and hence the complexity is simply $-\chi^-(S)=-\chi(S)$.

\subsection{A normal form}\label{subsec: normal form}
In this subsection, we focus on the case of $w\in G=A\star B$ and introduce a decomposition of any $w$-admissible surface into a \emph{simple normal form}, after possibly simplifying the surface first. This allows us to put specific structures on the surface to aid our Euler characteristic estimate.

The (simple) normal form we define below is similar to the one used in \cite{Chen:Kervaire}, both should generalize to graphs of groups/spaces and are adapted from the (simple) normal form in \cite{Chen:sclBS} of admissible surfaces in the context of stable commutator length.

Since we will focus on the case of a free product $G=A\star B$ and $w\in G$ not conjugate into $A$ or $B$, it is understood that the $w$-admissible surfaces we consider are \emph{relative to $\{A,B\}$}.

Let $X_A$ and $X_B$ be pointed $K(A,1)$ and $K(B,1)$ spaces respectively. Let $X=X_A\sqcup [-1,1] \sqcup X_B/\sim$, where the equivalence relation $\sim$ glues $-1$ (resp. $1$) to the base point of $X_A$ (resp. $X_B$). 
Then $X$ is a $K(G,1)$ space, which contains a point $\star$ corresponding to the midpoint $0$, and we treat it as the base point of $X$.
Let $\widehat{X}_A\defeq X_A\cup [-1,0]\subset X$, which is a copy of $K(A,1)$ based at $\star$, and similarly let $\widehat{X}_B\defeq X_B\cup[0,1]\subset X$; see Figure~\ref{fig: junctures_and_F}.

Fix any element $w\in G$ that is not conjugate into $A$ or $B$. 
Up to conjugation, we assume that $w$ is written as a cyclically reduced word $w=a_1 b_1\cdots a_n b_n$ for some $n\in\Z_+$, where each $a_i\neq id \in A$ and $b_i\neq id\in B$. Denote the \emph{word length} of $w$ as $|w|\defeq 2n$. Note that any such $w$ has infinite order.

We can represent $w$ as a map $f_w: S^1_w\to X$ from an oriented circle $S^1_w$, so that 
$$J_w\defeq f_w^{-1}(\star)\subset S^1_w$$ 
consists of $|w|$ points, which we refer to as the \emph{junctures}.
The junctures cut $S^1_w$ into $|w|$ arcs $\alpha_1, \beta_1,\ldots, \alpha_n, \beta_n$ in cyclic order, where each $\alpha_i$ (resp. $\beta_i$) is mapped by $f_w$ to a loop in $X$ supported on $\widehat{X}_A$ (resp. $\widehat{X}_B$) based at $\star$ representing $a_i\in A$ (resp. $b_i\in B$); see Figure~\ref{fig: junctures_and_F}.

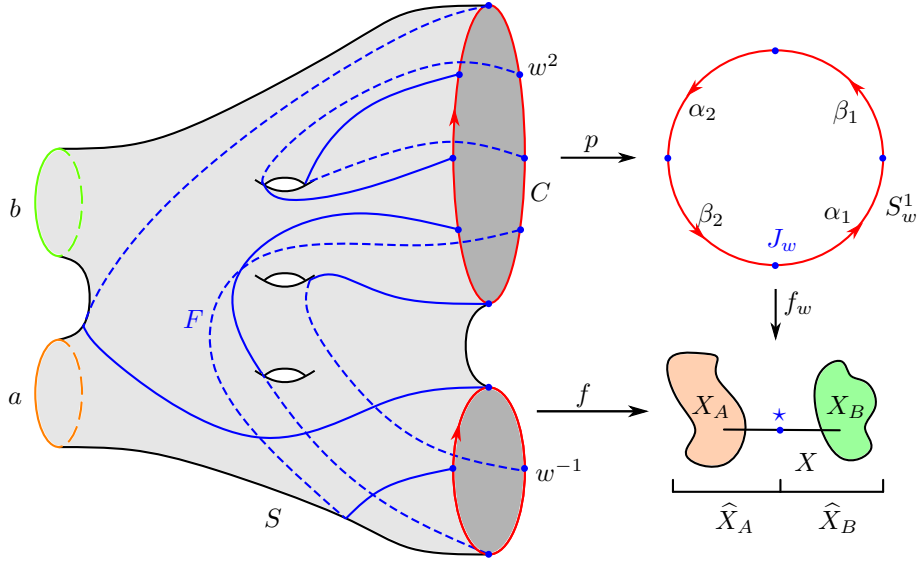
\begin{figure}
	\labellist
	\small \hair 2pt
	\pinlabel $b$ at -12 220
	\pinlabel $a$ at -12 100
	\pinlabel $S$ at 150 25
	\pinlabel $w^{-1}$ at 330 55
	\pinlabel $w^2$ at 320 310
    \pinlabel $C$ at 318 230
	
	\pinlabel $f$ at 345 102
    \pinlabel $p$ at 350 260
    \pinlabel $\color{blue}{F}$ at 100 150
	
	\pinlabel $\color{blue}{\star}$ at 468 90
    \pinlabel $\color{blue}{J_w}$ at 470 200
	\pinlabel $X$ at 485 60
	\pinlabel $S^1_w$ at 545 220
	\pinlabel $f_w$ at 480 160
    \pinlabel $X_A$ at 425 95
	\pinlabel $X_B$ at 510 95
	\pinlabel $\widehat{X}_A$ at 440 25
	\pinlabel $\widehat{X}_B$ at 505 25
    \pinlabel $\alpha_1$ at 505 215
    \pinlabel $\beta_1$ at 510 280
    \pinlabel $\alpha_2$ at 420 280
    \pinlabel $\beta_2$ at 425 215
    \endlabellist
	\centering
	\includegraphics[scale=0.6]{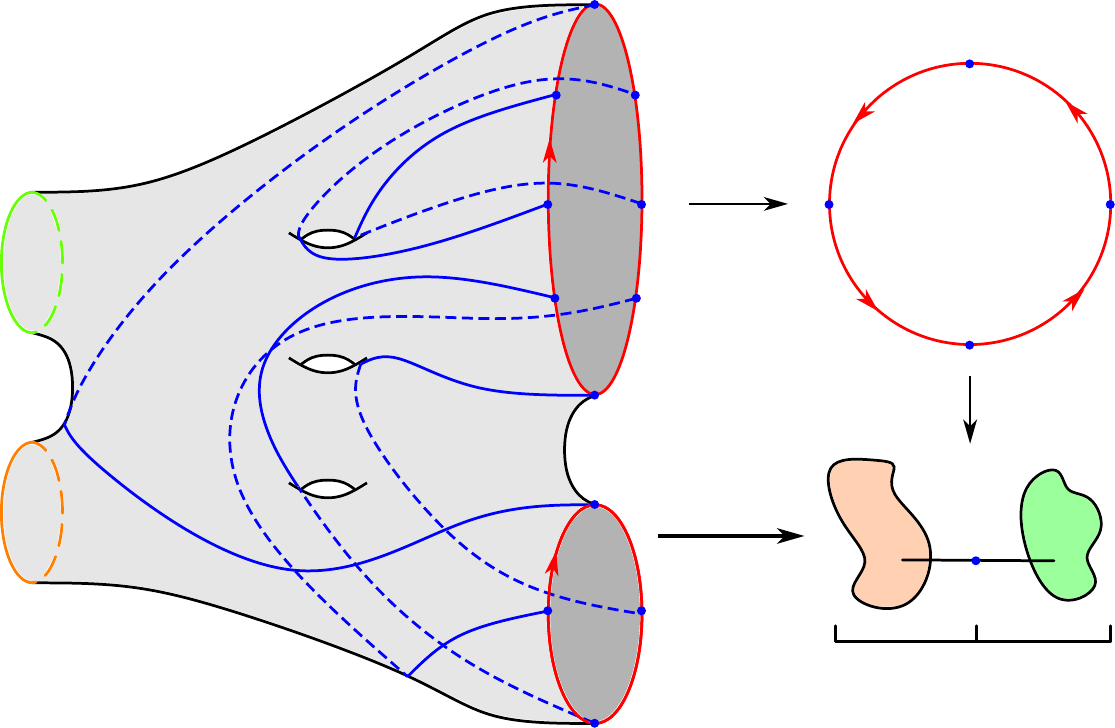}
	\caption{The set $J_w=f_w^{-1}(\star)$ on the circle $S^1_w$. The set of disjoint proper arcs $F=f^{-1}(X_C)$ is a set of embedded disjoint proper arcs in a reduced $w$-admissible surface $S$. The $w$-boundary component $C$ has a cover map $p$ to $S^1_w$ that pulls back $J_w$ to $J_C$.}\label{fig: junctures_and_F}
\end{figure}

For any $w$-admissible surface $(f,S)$, up to homotopy, we assume that the restriction of $f$ to each $w$-boundary component $C$ factors as $C\cong S^1\stackrel{p}{\to} S^1_w\stackrel{f_w}{\to} X$, where $p$ is a covering map whose degree $d$ is the degree of the component $C$. Then $p$ pulls back the set of junctures $J_w$ to a set $J_C\defeq p^{-1}(J_w)$ of $d|w|$ points on $C$, which we also refer to as junctures; see Figure~\ref{fig: junctures_and_F}. These junctures cut $C$ into $d|w|$ \emph{segments}, each is mapped to some $\alpha_i^{\pm 1}$ or $\beta_i^{\pm 1}$ under the map $p$, where the sign is $+1$ if and only if $C$ represents a positive power of $w$. 

We may also assume up to a homotopy that each $A$-boundary or $B$-boundary of $S$ has image under $f$ away from $\star$.

Since $\star$ is locally a submanifold of codimension one in $X$, and $f|_{\partial S}$ is already transverse to $\star$, we may assume that $f$ is transverse to $\star$, up to a homotopy of $f$ rel $\partial S$.
Then by transversality, $F\defeq f^{-1}(\star)$ is a properly embedded $1$-submanifold of $S$. That is, $F$ is a finite disjoint union of embedded proper arcs and embedded loops. 

\begin{definition}[reduced]
    We say a $w$-admissible surface $S$ is \emph{reduced} if each connected component of $S$ intersects the $w$-boundary, and $F=f^{-1}(\star)$ contains no loop, that is, $F$ is a finite disjoint union of embedded proper arcs; see Figure~\ref{fig: junctures_and_F}.
    In particular, there are no closed or disk components since $w$ has infinite order, so $\chi^-(S)=\chi(S)$ if $S$ is reduced.
\end{definition}

We can always simplify $S$ to a reduced one by the next lemma.

\begin{lemma}\label{lemma: reduced}
    If $(f,S)$ is a $w$-admissible surface in $X$, then there is a $w$-admissible surface $(f',S')$ with $\deg(S')=\deg(S)$ and $-\chi^-(S')\le -\chi^-(S)$ such that $S'$ is reduced. Moreover, $(f',S')$ is boundary-incompressible if $(f,S)$ is.
\end{lemma}
\begin{proof}
    First we discard components of $S$ that are disjoint from the $w$-boundary. This preserves $\deg(S)$ and boundary-incompressibility, and it does not increase $-\chi^-(S)$. 

    Now assume each component of $S$ witnesses the $w$-boundary. Then $S$ has no closed or disk components as $w$ has infinite order, so $\chi^-(S)=\chi(S)$.
    
    Suppose there is a loop $L\subset F$, which is embedded in $S$. 
    Let $c: D\to X$ be the constant map taking a closed disk $D$ to $\star\in X$. Cutting $S$ along $L$ creates two new boundary components, and we close them up by gluing in two copies of $D$ along the boundary. Denote the new surface by $S'$. 
    Note that $f|_L$ agrees with $c|_{\partial D}$, so we obtain a well-defined map $f':S'\to X$.

    Then $(f',S')$ is $w$-admissible with $\deg(S')=\deg(S)$ since boundary components of $S'$ exactly correspond to those of $S$, on which the map $f'$ agrees with $f$. 

    Denote by $\Sigma$ the component of $S$ containing $L$, which becomes $\Sigma'\subset S'$, which is either one or two components of $S'$ depending on whether $L$ separates $\Sigma$. Note that $-\chi(\Sigma')=-\chi(\Sigma)-2$ by Mayer--Vietoris. By our assumption above, $\Sigma$ is not closed, so $L$ cannot bound disks on both sides. 
    \begin{enumerate}
        \item If $L$ bounds a disk on one side, the process above creates a sphere component, and the other component is homeomorphic to $\Sigma$. In this case we delete the sphere component and redefine $S'$ to be the remaining surface (Figure~\ref{fig: compress}), in which case $\chi^-(\Sigma')=\chi(\Sigma')=\chi(\Sigma)=\chi^-(\Sigma)$.
        \item Otherwise, this does not create any sphere component and can create at most two disk components\footnote{Actually, it is also easy to show that this does not occur.}, so $\chi^-(\Sigma') \ge \chi(\Sigma')-2$. 
        Thus 
        $$-\chi^-(\Sigma')\le -\chi(\Sigma')+2=-\chi(\Sigma)=-\chi^-(\Sigma).$$
    \end{enumerate}
    Hence in any case we have $-\chi^-(S')\le -\chi^-(S)$.
    For each component of $\Sigma'$, homotope $f'$ slightly to push $f'(D)$ away from $\star$ to eliminate the loop $L$ from $F'=f'^{-1}(\star)$ so that $F'$ has fewer loops than $F$.
    
    It remains to see that $(f',S')$ inherits the boundary-incompressibility of $(f,S)$.
    Suppose $(f',S')$ is boundary-compressible due to an embedded pair of pants $P$. Denote the boundary components of $P$ as $L_1$, $L_2$ and $L_3$, where $L_1$ and $L_2$ are $w$-boundary components of $S'$ (coming from $S$) and $L_3$ is in the interior of $S'$. Up to an isotopy, we may shrink $P$ to a tubular neighborhood of $L_1\cup L_2\cup \gamma$ for some proper embedded arc $\gamma$ connecting $L_1$ and $L_2$. We can homotope $\gamma$ away from the new disks in the construction of $S'$, which does not change the homotopy class of $f'(\partial L_3)$; see Figure~\ref{fig: compress}. 
    So we may assume that $P$ is disjoint from the new disks and thus correspond to a pair of pants in $S$ that witnesses the boundary-compressibility of $(f,S)$.

    The process above strictly decreases the number of loop components in $F$, so by repeating the whole process above finitely many times we arrive at a desired reduced $(f',S')$.
\end{proof}

\begin{figure}
	\labellist
	\small \hair 2pt
	\pinlabel $S$ at 60 25
	\pinlabel $L_1$ at 160 160
	\pinlabel $L_2$ at 160 95
    \pinlabel $L$ at 80 145
    \pinlabel $P$ at 115 105
    
    \pinlabel $S'$ at 280 25
    \pinlabel $L_1$ at 380 160
	\pinlabel $L_2$ at 380 95
    \pinlabel $L$ at 305 145
    \pinlabel $D$ at 325 155
    \pinlabel $\gamma$ at 280 120
    \pinlabel $P$ at 335 105
    \endlabellist
	\centering
	\includegraphics[scale=0.8]{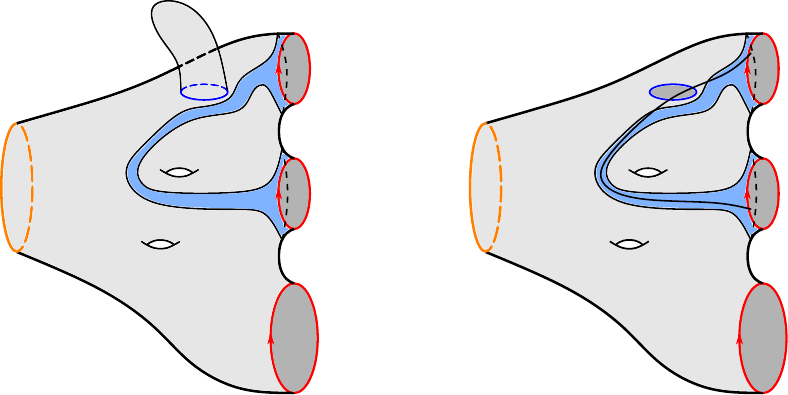}
	\caption{Compress a $w$-admissible surface $S$ along a loop $L\subset F$ to obtain $S'$. Push the proper arc $\gamma$ away from the disk $D$ bounding $L$ in $S'$ to isotope the pair of pants $P\simeq L_1\cup L_2\cup\gamma$ so that it ``lifts'' to $S$.}\label{fig: compress}
\end{figure}
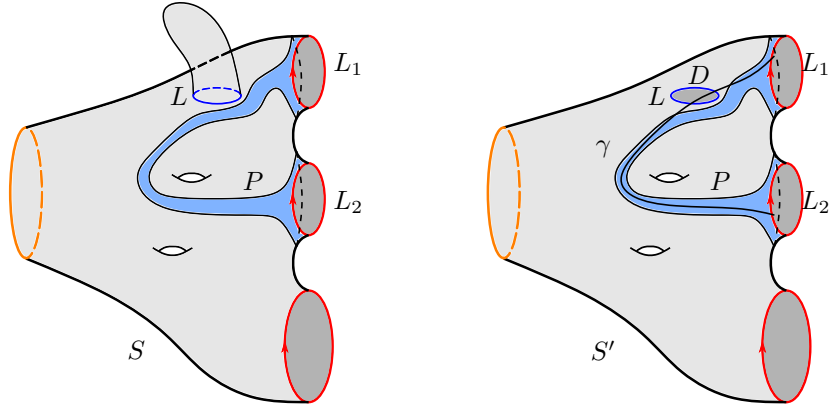

Now assume that $(f,S)$ is already a reduced $w$-admissible surface. In particular, 
$F=f^{-1}(\star)$ is a finite disjoint union of proper arcs, whose endpoints are necessarily junctures on the $w$-boundary of $S$. 
Note that any juncture is contained in exactly one arc in $F$ by disjointness.
Then $F$ cuts $S$ into two closed (as subsets) possibly disconnected subsurfaces $S_A\defeq f^{-1}(\widehat{X}_A)$ and $S_B\defeq f^{-1}(\widehat{X}_B)$ so that $S=S_A \cup_F S_B$; see the left of Figure~\ref{fig: simple}.

There are two possible kinds of boundary components of $S_A$:
\begin{enumerate}
    \item $A$-boundary components, exactly corresponding to the $A$-boundary components of $S$ (which are disjoint from $F$); see the orange loops on the right of Figure~\ref{fig: simple}; and
    \item \emph{polygonal boundary components}, each of which has an even number of sides, alternating between segments on some $w$-boundary of $S$ and proper arcs in $F$, the latter are called \emph{turns}; see the red and blue arcs on the right of Figure~\ref{fig: simple}.
\end{enumerate}

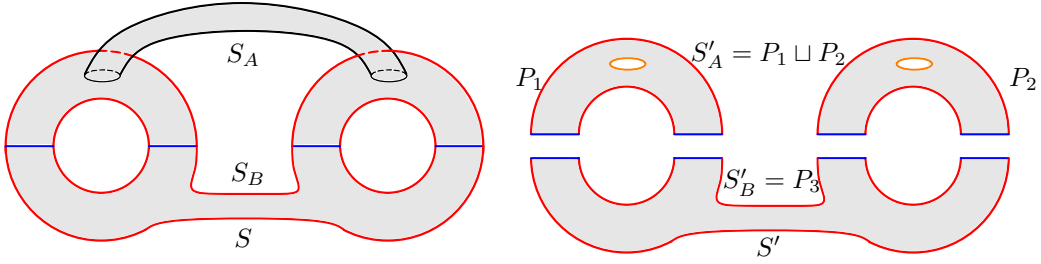
\begin{figure}
	\labellist
	\small \hair 2pt
	\pinlabel $S$ at 150 10
	\pinlabel $S_A$ at 150 125
	\pinlabel $S_B$ at 153 50
    
    \pinlabel $S'$ at 480 5
	\pinlabel $S'_A=P_1\sqcup P_2$ at 480 125
	\pinlabel $S'_B=P_3$ at 482 45
    
    \pinlabel $P_1$ at 330 110
    \pinlabel $P_2$ at 640 110
    
    \endlabellist
	\centering
	\includegraphics[scale=0.6]{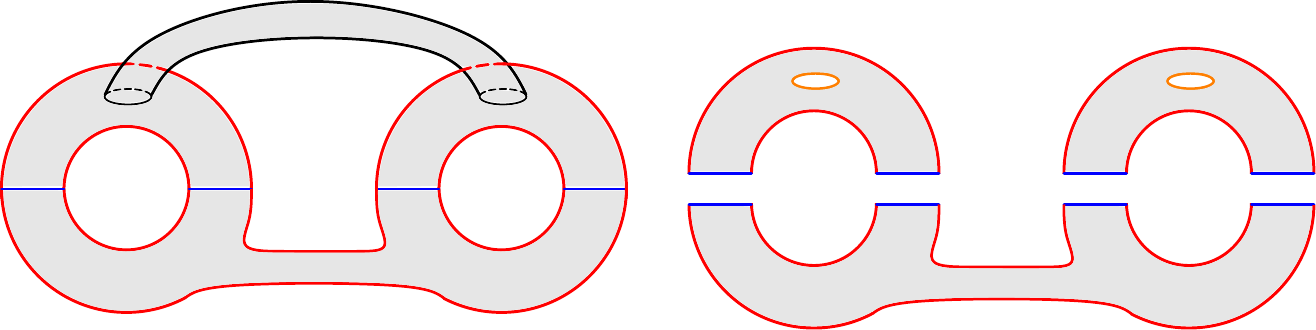}
	\caption{Left: The proper arcs in $F$ cut a reduced $w$-admissible surface $S$ into $S_A$ and $S_B$, and in this example $S_A$ has two polygonal boundary components and $S_B$ has one. Right: Simplifying $S$ as in Lemma~\ref{lemma: simple normal form} gives a $w$-admissible $S'$ in simple normal form, which is a union of a disk-piece $P_3$ and two annulus-pieces $P_1$ and $P_2$, where $d(P_1)=d(P_2)=2$, and $d(P_3)=4$.}\label{fig: simple}
\end{figure}

Since $f(S_A)\subset \widehat{X}_A$, each polygonal boundary component of $S_A$ is mapped to a loop in $\widehat{X}_A$, representing some conjugacy class in $A$, referred to as the \emph{winding class} of this polygonal boundary component. 
Note that the winding class is necessarily trivial if the polygonal boundary component bounds a disk component of $S_A$.

There is a similar structure on $S_B$ and we define the terminology similarly. The surface $S$ is obtained by gluing $S_A$ and $S_B$ along the turns.

\begin{definition}[(simple) normal form, disk-pieces, and annuli-pieces]\label{def: normal form}
    We call each connected component of $S_A$ or $S_B$ a \emph{piece}. 
    We refer to such a decomposition of $S$ into pieces as the \emph{normal form} of the reduced $w$-admissible surface $(f,S)$.
    We say the normal form is \emph{simple} if each piece has exactly one polygonal boundary component and is homeomorphic to either a disk or an annulus; see the right of Figure~\ref{fig: simple}.

    In the simple normal form, we refer to the two kinds of pieces as \emph{disk-pieces} and \emph{annulus-pieces} depending on the topological type. For each such piece $P$, define its \emph{valence} as $d(P)\defeq k$ if the unique polygonal boundary component has $2k$-sides (i.e. $k$ segments), where $k\in\Z_+$. Denote the Euler characteristic of such a piece by $\chi(P)$, which is $1$ (resp. $0$) if $P$ is a disk-piece (resp. annulus-piece).
\end{definition}

We can always simplify a $w$-admissible surface into one that admits a simple normal form. In particular, a piece in $S_A$ with more than one $A$-boundary component can be simplified.
\begin{lemma}\label{lemma: simple normal form}
    For any $w$-admissible surface $(f,S)$, there is another $w$-admissible surface $(f',S')$, such that $(f',S')$ has a simple normal form, with $-\chi^-(S')\le -\chi^-(S)$ and $\deg(S')= \deg(S)$. Moreover, $(f',S')$ is boundary-incompressible if $(f,S)$ is.
\end{lemma}
\begin{proof}
    By Lemma \ref{lemma: reduced}, we may assume that $(f,S)$ is reduced. 
    Decompose $S$ into the normal form by the process above. 
    Then each piece must contain at least one polygonal boundary component, since otherwise this piece must be a component of $S$ disjoint from the $w$-boundary.
    
    In particular, any piece homeomorphic to a disk must have a unique polygonal boundary component, thus it is a disk-piece.

    Consider any piece $P$ that is not a disk. 
    Then $\chi(P)\le 0$. Let us assume that $P\subset S_A$, and the other case is similar.
    For each polygonal boundary component, cut out a tubular neighborhood $N$ of it. Then $N$ is an annulus and we treat the new boundary component of $N$ as an $A$-boundary as it is mapped to a loop in $\widehat{X}_A$. This cuts $P$ into some $P'$ homeomorphic to $P$ together with finitely many annuli like $N$. 
    So $\chi(P')=\chi(P)\le 0$, and throwing $P'$ away does not increase $-\chi^-(S)$. Moreover, the operation above is performed away from the polygonal boundaries, so it does not affect the gluing of $S_A$ with $S_B$ and does not change the $w$-boundary; see Figure~\ref{fig: simple}.
    
    Applying this to each such piece $P$ in $S$, denote the new surface as $S'$ and define $f'$ as the restriction of $f$. This makes $(f',S')$ a reduced $w$-admissible surface with $-\chi^-(S')\le -\chi^-(S)$ and $\deg(S')= \deg(S)$. Moreover, the normal form of $S'$ consists of disk-pieces and annuli-pieces by construction.
    Since we are just deleting a subsurface of $S$ to obtain $S'$, boundary-incompressibility of $S$ implies that of $S'$.
\end{proof}

Finally, we give a formula to compute the Euler characteristic of a $w$-admissible surface $S$ in a simple normal form, which we will use in the next section for our estimates.

\begin{lemma}\label{lemma: Euler char formula}
    For any $w$-admissible surface $S$ in a simple normal form, we have
    $$
        -\chi^-(S)=-\chi(S)=\sum_P \left[\frac{d(P)}{2} - \chi(P)\right],
    $$
    where the summation is taken over all pieces $P$ in the decomposition of $S$.
\end{lemma}
\begin{proof}
    Any $S$ in simple normal form is reduced. 
    Hence $S$ has no closed or disk components (as $w$ has infinite order). Thus $\chi^-(S)=\chi(S)$.
    
    Now by Mayer--Vietoris, using the fact that $S_A$ and $S_B$ are disjoint unions of pieces $P$, we have 
    $$\chi(S)=\chi(S_A)+\chi(S_B)-\chi(F)=\sum_P \chi(P) - |F|,$$
    where $|F|$ denotes the number of proper arcs in $F$ (each having Euler characteristic $1$).
    
    Note that each piece witnesses $d(P)$ proper arcs in $F$, namely the turns in $P$, and each proper arc in $F$ is counted exactly twice, by a piece in $S_A$ and one in $S_B$. Thus $|F|=\sum_P d(P)/2$. Combining this with the formula above, the desired equation 
    follows.
\end{proof}

\section{The spectral gap property via stackings}\label{sec: topol argument}
The purpose of this section is to prove Theorem~\ref{thmA: spectral gap from stacking}, which we restate as the following stronger version, allowing $A$ and $B$ to be uncountable. This is not essential, since a $w$-admissible surface in simple normal form only sees the countable subgroups $A'$ and $B'$. So we will assume $A$ and $B$ to be countable in the proof below.


\begin{theorem}\label{thm: spectral gap from stacking}
    For any free product $G=A\star B$ and a cyclically reduced word $w=a_1b_1\ldots a_n b_n\in A\star B$ with $n\in\Z_+$, $a_i\in A\setminus\{id\}$ and $b_i\in B\setminus\{id\}$. If there is a relative stacking of $w$ in $G'=A'\star B'$, where $A'$ (resp. $B'$) is the subgroup generated by $a_1,\ldots,a_n$ (resp. $b_1,\dots,b_n$),
    then for any boundary-incompressible $w$-admissible surface $S$ in $G$, we have
    $$-\chi^-(S)\ge \deg(S).$$
\end{theorem}

We use the setup as in Section~\ref{subsec: normal form}. 
By Lemma~\ref{lemma: simple normal form}, it suffices to prove the theorem for any $(f,S)$ that is in simple normal form. For our estimate of Euler characteristic, we need to put some additional structure on the pieces using the relative stacking.

A \emph{prefix} $u$ of $w$ is a subword of $w$ of the form $u=a_1b_1\ldots a_i$ or $u=a_1b_1\ldots a_ib_i$ for for some $1\le i\le n$. There are exactly $|w|=2n$ of them.

As in Section~\ref{subsec: normal form}, we view $w$ as a map $f_w: S^1_w\to X$ so that the set $J_w=f_w^{-1}(\star)$ of junctures divides 
the oriented circle $S^1_w$ into $|w|=2n$ oriented arcs 
$\alpha_1,\beta_1,\ldots,\alpha_n,\beta_n$ in cyclic order.
The map $f_w$ takes each arc $\alpha_i$ (resp. $\beta_i$) to a loop based at $\star$ representing $a_i\in A$ (resp. $b_i\in B$).
Note that the junctures correspond to prefixes $u$ of $w$, by thinking of a juncture as where the prefix $u$ ends, where the juncture between $\beta_n$ and $\alpha_1$ corresponds to the full word $w$.

Recall that a relative stacking of $w$ gives a right action of $G$ on $\R$ and a special point $x$ whose trajectory $\Omega(w,x)$ is stable (i.e. this is a multiset without repeated elements and also $x\cdot w=x$), where each point in the trajectory is $x\cdot u$ for a nonempty prefix $u$ of $w$. Hence a relative stacking gives an assignment $\lambda: J_w\to \R$, where the value of the juncture corresponding to a prefix $u$ is $x\cdot u$. We record the data of a relative stacking into properties of this assignment $\lambda$ as follows:

\begin{lemma}\label{lemma: lambda}
    Given a relative stacking of $w$, the assignment $\lambda: J_w\to \R$ is \emph{injective}, with the property that, for each arc $\alpha_i$ (resp. $\beta_i$) on $S^1_w$ going from a juncture $j$ to another juncture $j'$, we have $\lambda(j)\cdot a_i=\lambda(j')$ (resp. $\lambda(j)\cdot b_i=\lambda(j')$), for the action given by the relative stacking.
\end{lemma}
\begin{proof}
    The injectivity of $\lambda$ follows from the fact that the elements in $\Omega(w,x)$ are distinct. The other property is by definition except for the arc $\alpha_1$: It goes from the juncture $j$ corresponding to $w$ to the juncture $j'$ corresponding to the prefix $a_1$, but note that $\lambda(j)=x\cdot w=x$, so $\lambda(j)\cdot a_1 =x \cdot a_1=\lambda(j')$ as desired.
\end{proof}
The properties of $\lambda$ described above involve the right action from the relative stacking, and essentially encapsulate the stability of $\Omega(w,x)$.
In what follows, we will just work with the assignment $\lambda$ instead of the relative stacking.

In this form, interpreted topologically, this is analogous to the notion of stacking in \cite{LouderWilton:stacking}. In fact, one can generalize our notion of relative stacking to the setting of graphs of groups with trivial edge groups. If the vertex groups are also trivial, then it is equivalent to stacking in the sense of \cite{LouderWilton:stacking}.
Our definition is inspired by but different from the notion of relative stacking in \cite{Millard}.

Recall from Section~\ref{subsec: normal form} that the restriction of the map $f$ to each $w$-boundary component $C$ of $S$ of some degree $d\ge 1$ factors as a map of the form $C\cong S^1\stackrel{p}{\to} S^1_w\stackrel{f_w}{\to} X$ for a covering map $p$ of degree $d$. The map $p$ pulls back the junctures $J_w$ on $S^1_w$ to junctures $J_C\defeq p^{-1}(J_w)$ on $C$, which cut $C$ into $d|w|$ segments, each labeled by some $a_i^s$ or $b_i^s$, where $s=\pm1$ records whether $p$ is orientation-preserving when the $C$ is equipped with the orientation induced from $S$.

The map $p$ pulls back the assignment $\lambda$ to an assignment $\lamhat_C: J_C\to \R$ defined as $\lamhat_C\defeq \lambda \circ p$, which is $|w|$-periodic when we go around $J_C$ in the cyclic order. Define $J_S$ as the disjoint union of $J_C$ over all $w$-boundary components $C$, and define $\lamhat: J_S\to \R$ component-wise using $\lamhat_C$. Note that $|J_S|=|w|\deg(S)$.

As $(f,S)$ is in simple normal form, it is reduced and $F=f^{-1}(\star)$ is a finite disjoint union of embedded proper arcs, which correspond to turns on the pieces. Each such arc $\gamma$ connects two junctures $j,j'\in J_S$, and we \emph{orient} $\gamma$ so that it points from $j$ to $j'$ if $\lamhat(j)>\lamhat(j')$ and the other way around if $\lamhat(j)<\lamhat(j')$. The basic but important observation below shows that this defines an orientation on each arc $\gamma\subset F$ whenever $(f,S)$ is boundary-incompressible.

\begin{lemma}
    If $(f,S)$ is boundary-incompressible and in simple normal form, then for any arc $\gamma \subset F$ connecting junctures $j,j'\in J_S$, we have $\lamhat(j)\neq\lamhat(j')$.
\end{lemma}
\begin{proof}
    Let $C$ and $C'$ be the $w$-boundary components of $S$ that $j$ and $j'$ belong to, respectively. 
    Suppose $\lamhat(j)=\lamhat(j')$. Since the map $\lambda$ is injective by Lemma~\ref{lemma: lambda}, the junctures $j$ and $j'$ must correspond to the same juncture $j_*\in J_w$ on $S^1_w$. Assume that $j_*$ is the one between $\alpha_i$ and $\beta_i$ for some $1\le i\le n$. The case where $j_*$ is between some $\beta_i$ and $\alpha_{i+1}$ is similar.
    
    With the induced orientation from $S$, there is a segment $s_1$ on $C$ (resp. $s'_1$ on $C'$) ending at $j$ (resp. $j'$) and a segment $s_2$ on $C$ (resp. $s'_2$ on $C'$) starting at $j$ (resp. $j'$); see the left of Figure~\ref{fig: coorientation}. Suppose $C$ (resp. $C'$) represents the conjugacy class of $w^n$ for some $n\neq0 \in\Z$ (resp. $w^{n'}$ for some $n'\neq0 \in\Z$). There are two cases for the labels on $s_1$ and $s_2$ depending on the sign of $n$, and similarly on the $C'$ side:
    \begin{enumerate}
        \item If $n>0$ (resp. $n'>0$), then $s_1$ (resp. $s'_1$) is labeled by $a_i$ and $s_2$ (resp. $s'_2$) by $b_i$;
        \item If $n<0$ (resp. $n'<0$), then $s_1$ (resp. $s'_1$) is labeled by $b_i^{-1}$ and $s_2$ (resp. $s'_2$) by $a_i^{-1}$.
    \end{enumerate}

    We now observe that $n$ and $n'$ must have opposite signs, and in particular $C\neq C'$. 
    In fact, recall that $F$ cuts $S$ into two subsurfaces $S_A$ and $S_B$, each of which consists of pieces.
    The arc $\gamma$ appears as turns in exactly two pieces $P_A\subset S_A$ and $P_B\subset S_B$.
    Up to swapping $j$ and $j'$, let us assume that the turn on $P_A$ with the induced orientation from $P_A$ is from $j$ to $j'$; see the left of Figure~\ref{fig: coorientation}. 
    Then $s_1$ and $s'_2$ lie in $P_A$ and must be labeled by something in the group $A$, so we must have $n>0$ and $n'<0$ based on the possibilities listed above.
    
    Let $N$ be a tubular neighborhood of $C\cup \gamma\cup C'$, which is homeomorphic to a pair of pants with boundary components $C$, $C'$ and $L$ for some embedded loop $L$ in $S$. Note that $f$ maps $\gamma$ to $\star$, so it follows that $L$ with the induced orientation from $N$ represents the conjugacy class of $w^{-(n'+n)}$, which implies that $(f,S)$ is boundary-compressible; see Definition~\ref{def: boundary-incompressible} and take $m=-n'$. This contradicts our assumption and thus we cannot have $\lamhat(j)=\lamhat(j')$.
\end{proof}

\begin{figure}
	\labellist
	\small \hair 2pt
	\pinlabel $j$ at 70 45
    \pinlabel $j'$ at 75 95
    \pinlabel $s_1(a_i)$ at 45 40
    \pinlabel $s_2'(a_i^{-1})$ at 50 100
    \pinlabel $C$ at 30 25
    \pinlabel $C'$ at 35 115
    \pinlabel $s_2(b_i)$ at 100 40
    \pinlabel $s_1'(b_i^{-1})$ at 105 105
    \pinlabel $L$ at 40 72
    \pinlabel $L$ at 105 70
    \pinlabel $N$ at 60 72
    \pinlabel $\gamma$ at 80 72
    \pinlabel $P_A$ at 10 72
    \pinlabel $P_B$ at 135 70
    
    \pinlabel $j_1^+$ at 265 5
    \pinlabel $j_2^-$ at 315 5
    \pinlabel $j_2^+$ at 355 45
    \pinlabel $j_3^-$ at 355 95
    \pinlabel $j_3^+$ at 315 135
    \pinlabel $j_4^-$ at 265 135
    \pinlabel $j_4^+$ at 225 95
    \pinlabel $j_1^-$ at 225 45
    \pinlabel $\gamma_2$ at 290 10
    \pinlabel $\gamma_3$ at 350 70
    \pinlabel $\gamma_4$ at 290 130
    \pinlabel $\gamma_1$ at 232 70
    
    \pinlabel $s_1$ at 260 40
    \pinlabel $s_2$ at 320 40
    \pinlabel $s_3$ at 320 100
    \pinlabel $s_4$ at 260 100
    \pinlabel $P$ at 260 70
    
	\pinlabel $j_5^-$ at 375 45
    \pinlabel $j_5^+$ at 445 45
    \pinlabel $j_6^-$ at 445 95
    \pinlabel $j_6^+$ at 375 95
    \pinlabel $s_5$ at 410 50
    \pinlabel $s_6$ at 410 90
    \pinlabel $P'$ at 400 70
    
    \endlabellist
	\centering
	\includegraphics[scale=0.9]{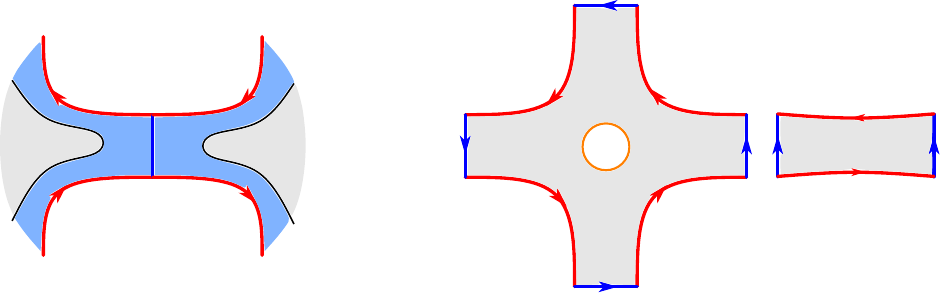}
	\caption{Left: The local structure of an arc $\gamma\in F$ connecting two junctures $j\in J_C$ and $j'\in J_{C'}$ that correspond to the same $j_*\in J_w$. A neighborhood $N$ of $C\cup C\cup \gamma$ witnesses boundary-compressibility.
    Right: Two pieces $P$ and $P'$ in the simple normal form of some $S$, marked with $\lamhat$-orientations. In this example, $\sc(P)=0$ and $\sc(P')=2$.}\label{fig: coorientation}
\end{figure}
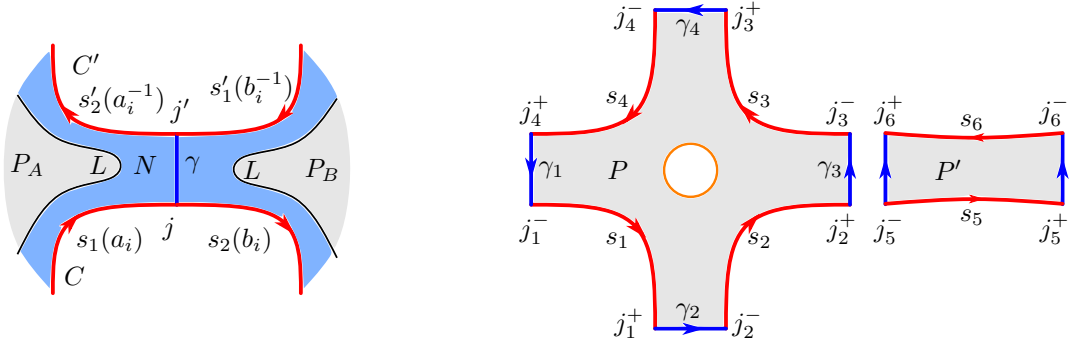

With the lemma above, this defines an orientation on each arc in $F$ if $S$ is boundary-incompressible. Refer to this as the \emph{$\lamhat$-orientation}. Based on this, we make the following definition:
\begin{definition}
    For each segment $s$ on the $w$-boundary of $S$ connecting two junctures $j_1$ and $j_2$, let $\gamma_1$ and $\gamma_2$ be the arcs in $F$ that contain $j_1$ and $j_2$ respectively.
    There are four cases for the $\lamhat$-orientations on $\gamma_1$ and $\gamma_2$:
    \begin{enumerate}
        \item $\gamma_1$ points towards $j_1$, and $\gamma_2$ points towards $j_2$;\label{item: ++}
        \item $\gamma_1$ points away from $j_1$, and $\gamma_2$ points away from $j_2$;\label{item: --}
        \item $\gamma_1$ points towards $j_1$, and $\gamma_2$ points away from $j_2$; and\label{item: +-}
        \item $\gamma_1$ points away from $j_1$, and $\gamma_2$ points towards $j_2$.\label{item: -+}
    \end{enumerate}
    We say $s$ is \emph{consistent} if we have case (\ref{item: ++}) or (\ref{item: --}), and we say $s$ is \emph{inconsistent} otherwise.
    For example, in the right part of Figure~\ref{fig: coorientation}, the segment $s_i$ is inconsistent for $1\le i\le 4$ and is consistent for $i= 5,6$. 
\end{definition}

The lemma below is a key observation that helps us estimate the number of (in)consistent segments on the $w$-boundary of $S$: Each copy of $w$ (or $w^{-1}$) on $\partial S$ witnesses at least two inconsistent arrows.
\begin{lemma}\label{lemma: segment count}
    For any boundary-incompressible $w$-admissible surface $S$ in simple normal form, the total number of inconsistent segments is at least $2\deg(S)$, thus the total number of consistent segments is no more than $(|w|-2)\deg(S)$.
\end{lemma}
\begin{proof}
    It suffices to prove the lower bound on the total number of inconsistent segments, as the total number of segments is $|w|\deg(S)$.

    Since $\lambda:J_w\to \R$ is injective by Lemma~\ref{lemma: lambda}, there is a unique juncture $j_M\in J_w$ (resp. $j_m\in J_w$) for which $\lambda(j_M)\geq \lambda(j)$ (resp. $\lambda(j_m)\leq \lambda(j)$) for all $j\in J_w$. Clearly $j_M\neq j_m$.

    For each $w$-boundary component $C$ of degree $d$, there are exactly $d$ junctures in $J_C$ corresponding to $j_M$ (resp. $j_m$), which we call \emph{maximal} (resp. \emph{minimal}). Moreover, the $d$ maximal and $d$ minimal junctures sit in an alternating way on $C$.

    For each maximal (resp. minimal) juncture $j$, the adjacent arc $\gamma\subset F$ must have $\lamhat$-orientation pointing away from (resp. towards) $j$. Thus between each adjacent pair of maximal and minimal junctures, there is at least one segment that is inconsistent. Thus we must have at least $2d$ inconsistent segments on each $w$-boundary component of degree $d$. Taking the sum over all $w$-boundary components completes the proof.
\end{proof}

Another key observation (Lemma~\ref{lemma: sign changes}) is to count the number of consistent segments in each piece $P$. Recall that the valence $d(P)$ of a (disk- or annulus-) piece $P$ is the number of segments on the unique polygonal boundary component of $P$. 
Among these $d(P)$ segments, denote the number of consistent segments as $\sc(P)$, the number of \emph{sign changes} in $P$.

The proof of Lemma~\ref{lemma: sign changes} below justifies the terminology. This is where we make use of the action from the relative stacking.

\begin{lemma}\label{lemma: sign changes}
    For any (disk- or annulus-) piece $P$, the number of sign changes $\sc(P)\in\Z_{\ge0}$ is even. Moreover, $\sc(P)>0$ for each disk-piece $P$.
\end{lemma}
\begin{proof}
    Equip the polygonal boundary of $P$ with the orientation induced from $P$.
    Recall that the polygonal boundary of $P$ alternates with turns and segments, so we can enumerate them in the cyclic order as
    $\gamma_1,s_1,\ldots, \gamma_k,s_k$,
    where $k=d(P)\in\Z_+$ is the valence of $P$, each $\gamma_i\subset F$ is a turn and each $s_i$ is a segment on some $w$-boundary component of $S$.
    Denote the endpoints of each $s_i$ as $j_i^-$ and $j_i^+$, which are junctures on some $w$-boundary, so that $s_i$ goes from $j_i^-$ to $j_i^+$. Then each $\gamma_i$ with the topological orientation (induced from $P$) goes from $j_{i-1}^+$ to $j_i^-$, where indices are taken mod $k$ and similarly below; see the piece $P$ in the right part of Figure~\ref{fig: coorientation}.
    
    The topological orientation on each $\gamma_i$ may or may not agree with the $\lamhat$-orientation defined earlier. Assign a number $\sign_i\in\{1,-1\}$ for each $1\le i\le k$ so that $\sign_i=1$ if and only if these two orientations on $\gamma_i$ agree.
    
    For each segment $s_i$, the topological orientation on $\gamma_i$ points towards $j_i^-$, and the one of $\gamma_{i+1}$ points away from $j_i^+$. Thus $s_i$ is consistent if and only if $\sign_{i-1}=-\sign_i$.
    
    Thus the number $\sc(P)$ of consistent segments is equal to the number of $i\in\{1,\ldots, k\}$ with $\sign_{i-1}=-\sign_i$; see the pieces $P$ and $P'$ on the right of Figure~\ref{fig: coorientation} as examples.
    This is why we refer to it as the number of sign changes in $P$.
    Then clearly $\sc(P)$ is even.

    Next we show $\sc(P)\neq0$ if $P$ is a disk-piece. 
    To simplify the notation in the proof below, let us assume that $P$ lies in $S_A$, and the other case works by symmetry.
    Then each segment $s_i$ is labeled by some element $\varphi_i\in A$, where $\varphi_i=a_t^{\pm1}$ for a letter $a_t$ in the cyclically reduced expression of $w$, where $1\le t\le n$.
    By Lemma~\ref{lemma: lambda} and the definition of $\lamhat$, we know $$\lamhat(j_i^-)\cdot \varphi_i = \lamhat(j_i^+).$$
    Moreover, since $P$ is a disk-piece, the winding class is trivial, so
    $$\varphi_1\varphi_2\cdots \varphi_k=id.$$
    
    Suppose $\sc(P)=0$. Then we have either $\forall i, \sign_i\equiv 1$ or $\forall i, \sign_i\equiv -1$. Assume without loss of generality that $\sign_i\equiv 1$ for all $i$, as the other case is similar; see the piece $P$ (assuming it is a disk) on the right of Figure~\ref{fig: coorientation} as an illustration.
    That is, for each $i$, the $\lamhat$-orientation of each $\gamma_i$ goes from $j_{i-1}^+$ to $j_i^-$, meaning $$\lamhat(j_{i-1}^+)>\lamhat(j_{i}^-).$$
    
    Then we have $\lamhat(j_k^+)>\lamhat(j_{1}^-)$ by taking $i=1$. As $\varphi_1$ acts by an orientation-preserving homeomorphism of $\R$, we have
    $$\lamhat(j_k^+)\cdot \varphi_1 > \lamhat(j_{1}^-)\cdot \varphi_1 =\lamhat(j_{1}^+)>\lamhat(j_2^-).$$
    Then applying $\varphi_2$ to the inequality $\lamhat(j_k^+)\cdot \varphi_1>\lamhat(j_2^-)$, we similarly obtain
    $$\lamhat(j_k^+)\cdot \varphi_1\varphi_2 > \lamhat(j_2^-)\cdot \varphi_2=\lamhat(j_2^+)>\lamhat(j_3^-).$$
    Inductively, this shows that, for all $1\le \ell\le k$, we have
    $$\lamhat(j_k^+)\cdot \varphi_1\cdots \varphi_\ell > \lamhat(j_{\ell}^+)>\lamhat(j_{\ell+1}^-).$$
    Taking the first inequality with $\ell=k$ and using $\varphi_1\varphi_2\cdots \varphi_k=id$, we observe that
    $$\lamhat(j_k^+)=\lamhat(j_k^+)\cdot \varphi_1\cdots \varphi_k>\lamhat(j_{k}^+),$$
    which gives a contradiction.
\end{proof}

\begin{corollary}\label{cor: sc estimate}
    For any disk- or annulus-piece $P$, we have
    $$\sc(P)\ge 2\chi(P).$$
\end{corollary}
\begin{proof}
    If $P$ is an annulus-piece, then $2\chi(P)=0\le \sc(P)$. 
    If $P$ is a disk-piece, then $\chi(P)=1$, and $\sc(P)\ge2$ by Lemma~\ref{lemma: sign changes}.
\end{proof}

Now we are in a place to prove Theorem~\ref{thm: spectral gap from stacking}.
\begin{proof}[Proof of Theorem~\ref{thm: spectral gap from stacking}]
    By Lemma~\ref{lemma: simple normal form}, we may assume that $S$ is in simple normal form. In particular, $\chi^-(S)=\chi(S)$.
    Using the notions above, note that for each piece $P$, the valence $d(P)$ counts the total number of all segments on its polygonal boundary, among which $\sc(P)$ of them are consistent segments. 
    We know the sum of $d(P)$ over all pieces $P$ in $S$ is $|w|\deg(S)$, and the sum of $\sc(P)$ is at most $(|w|-2)\deg(S)$ by Lemma~\ref{lemma: segment count}.
    Combining this with 
    Lemma~\ref{lemma: Euler char formula}
    and Corollary~\ref{cor: sc estimate}, we obtain the following estimate of $-\chi(S)$:    
    \begin{align*}
        -\chi(S)&=\sum_P \left[ \frac{d(P)}{2}-\chi(P)\right]\\
        &\ge \frac{1}{2}\sum_P \left[ d(P)-\sc(P)\right]\\
        &\ge\frac{1}{2} \left[|w|\deg(S)-(|w|-2)\deg(S)\right]\\
        &=\deg(S).\qedhere
    \end{align*}
\end{proof}

The proof presented here is similar in spirit to the one in \cite{DH91}, but the novelty of our approach is the usage of $\lamhat$ as our labels of junctures, and its construction using actions on the line as presented in the next section. Besides, one can also phrase the topological argument above in terms of the LP duality method in \cite[Section 4]{Chen:Kervaire}, where the cost function here is defined using the $\lamhat$-orientation. It can also be phrased using angle structures similar to the method used in \cite{Marchand}.

\section{The existence of relative stackings}\label{sec: rel stacking}
All actions in this section will be right actions, and for homeomorphisms $\lambda_1:X\to Y$ and $\lambda_2:Y\to Z$, the composition will be denoted as $\lambda_1\lambda_2$, rather than $\lambda_2\circ \lambda_1$.
A classical result states that a countable group $G$ admits a faithful action by orientation-preserving homeomorphisms on the real line if and only if it is right-orderable. Given a right-order on $G$, one may construct such an action on $\R$ by first constructing a suitable order-embedding $G\to \R$ so that the natural right action on the image extends to an action by orientation- homeomorphisms on $\R$ \cite{GOD}. This is the so called \emph{dynamical realization} of the order. Conversely, given an embedding $G\to \textup{Homeo}^+(\R)$, one can extract a right-invariant order on $G$ as follows. Choose an enumeration $(x_n)_{n\in \N}$ of $\Q$. Then $f<g$ if and only if for the smallest $n\in \N$ such that $x_n\cdot f\neq x_n\cdot g$, it holds that $x_n\cdot f<x_n\cdot g$. This is clearly a right invariant total order.

Consider two countable right-orderable groups $A$ and $B$.
Denote $G=A\star B$.
Consider an action $\sigma:G\to \textup{Homeo}^+(\R)$ (not necessarily faithful),
and a word $w=a_1b_1\ldots a_nb_n$ such that $a_i\in A\setminus \{id\}, b_i\in B\setminus \{id\}$. 
Recall that a word $w_1$ is a \emph{prefix} of $w$ if it is a nonempty word of the form $a_1b_1\ldots a_i$ or $a_1b_1\ldots a_ib_i$ for $1\leq i\leq n$.
This is a \emph{proper prefix} if it does not equal $w$.
An ordered pair of nonempty words $(w_1,w_2)$ in $G$ is a \emph{prefix pair} if $w_1$ and $w_1w_2$ are both prefixes of $w$ (note that $w_1$ will be a proper prefix of $w$ but it may be the case that $w_1w_2=w$). 
Recall from Definition~\ref{def: relative stacking} that given $x\in \R$, we define the \emph{trajectory} of $x$ under $w$ as the multiset of points: $$\Omega(w,x)=\{x\cdot \sigma(w_1)\mid w_1\text{ is a prefix of }w\}.$$
The trajectory of $x$ under $w$ is \emph{stable} if the multiset $\Omega(w,x)$ has no element that occurs more than once, and $x\cdot \sigma(w)=x$.

Let $\{I_n\}_{n\in X}$ be a collection of pairwise disjoint nonempty open intervals in $\R$ for some countable set $X$ which may be finite or infinite. Let $\tau_n:H\to \textup{Homeo}^+(I_n), n\in X$ be actions of some group $H$.
We can define a \emph{diagonal-product action}: $$\tau:H\to \textup{Homeo}^+(\R)\qquad \tau:=\circledast_{n\in X}\tau_n$$
as the action where each group element fixes each point in $\R\setminus \bigcup_{n\in X}I_n$, and: $$x\cdot \tau(\alpha)=x\cdot \tau_n(\alpha)\qquad \text{ whenever }x\in I_n, \alpha\in H$$
We define the following variation which we also call the diagonal-product, for notational simplicity.
Consider an action $\sigma: H\to \textup{Homeo}^+(\R)$ such that $\sigma$ pointwise fixes a nonempty open interval $I$.
Consider another action $\tau:H\to \textup{Homeo}^+(I)$. Then we define $\sigma\circledast \tau$ also as the diagonal-product action which agrees with $\sigma$ on $\R\setminus I$ and with $\tau$ on $I$. 

Given actions $\sigma: A\to \textup{Homeo}^+(\R)$ and $\tau: B\to \textup{Homeo}^+(\R)$, we define the action
$$\eta:A\star B\to \textup{Homeo}^+(\R)\qquad \eta\defeq \langle \sigma, \tau\rangle$$ as the action generated by these (using the defining property of free products).
We say that $\eta$ is \emph{generated by the actions} $\sigma$and $\tau$, since it is generated by the images of these actions in $\textup{Homeo}^+(\R)$.
Note that $\eta$ may not be faithful, even if $\sigma$ and $\tau$ are.

The purpose of this section is to prove Theorem \ref{thmA: Stacking}. 
If $w$ is a proper power, it is easy to see that a relative stacking cannot exist. So we focus on the hard direction, which is reformulated below.

\begin{theorem}\label{thm:StackingsRF}
Let $A,B$ be countable right-orderable groups and $G=A\star B$. Consider a cyclically reduced word 
$$w=a_1b_1\ldots a_nb_n\in G, \qquad a_i\in A\setminus \{id\}, b_i\in B\setminus \{id\}.$$
Assume that $w$ is not a proper power.
Then there exists an action $\sigma:G\to \textup{Homeo}^+(\R)$ and some $x\in \R$ such that $\Omega(w,x)$ is stable.
\end{theorem}

We remark that in Theorem \ref{thm:StackingsRF}, we do not require the action $\sigma$ to be faithful. However, from such a $\sigma$, one can easily obtain an action satisfying the same conditions that is faithful. Let $\sigma_1:G\to \textup{Homeo}^+(0,1)$ be an action obtained by conjugating $\sigma$ via some homeomorphism $\nu:\R\to (0,1)$, and let $\sigma_2:G\to \textup{Homeo}^+(1,2)$ be a faithful action. Note that the latter exists since free products of left orderable groups are left orderable \cite{GOD}. Then the diagonal-product $\sigma_1\circledast \sigma_2$ is a faithful action of $G$ that satisfies the requirements of Theorem \ref{thm:StackingsRF}, as witnessed by the image $\nu(x)$ of the point $x$ which witnessed this for $\sigma$.

The following proposition will be a key step.

\begin{proposition}\label{Lem:Stackings}
Let $A,B$ be countable right-orderable groups and $G=A\star B$. Consider a cyclically reduced word 
$$w=a_1b_1\ldots a_nb_n\in A\star B, \qquad a_i\in A\setminus \{id\}, b_i\in B\setminus \{id\}.$$
Assume that $w$ is not a proper power.
Let $(w_1,w_2)$ be a prefix pair for $w$.
Then there exists an action of $\sigma:G\to \textup{Homeo}^+(\R)$ and some $x\in \R$ such that:
\begin{enumerate}
\item $x\cdot \sigma(w)=x$.
\item $x\cdot \sigma(w_1w_2)\neq x\cdot \sigma(w_1)$.
\end{enumerate}
\end{proposition}

First, we will see how we can use Proposition \ref{Lem:Stackings} to supply the proof of Theorem~\ref{thm:StackingsRF}.
The main ingredient in this will be Lemma~\ref{Lem:System} below, for which we need a few definitions.
Let $H$ be a countable group.
Given a fixed variable $y$, we shall consider \emph{systems of equations and inequations over $H$}, or for brevity simply a \emph{system over $H$}. Each equation in such a system will be of the form $y\cdot \alpha=y$ and each inequation will be of the form $y\cdot \beta\neq y$ where $\alpha,\beta\in H$ are fixed group elements. 
A solution to the system will be a pair $(\tau,x)$ where $\tau:H\to \textup{Homeo}^+(\R)$ is an action (not necessarily faithful) and $x\in \R$, so that:
\begin{enumerate}
\item For each equation $y\cdot \alpha=y$ in the system it holds that $x\cdot \tau(\alpha)=x$
\item For each inequation $y\cdot \beta\neq y$ in the system, it holds that $x\cdot \tau(\beta)\neq x$.
\end{enumerate}
Such a system is called a \emph{finite system} if there are finitely many such equations and inequations.
Given a system $\Lambda$, we denote by $\Lambda^=$ the set of equations in $\Lambda$ and by $\Lambda^{\neq}$ the set of inequations in $\Lambda$.
Given a finite collection of finite systems $\Lambda_1,\ldots ,\Lambda_n$, we define a new system
$$\Xi(\Lambda_1,\ldots, \Lambda_n):=(\bigcup_{1\leq i\leq n}\Lambda_i^{\neq })\cup (\bigcap_{1\leq i\leq n}\Lambda_i^=).$$ 

For example, consider the Baumslag--Solitar group $\mathrm{BS}(1,2)=\langle f,g\mid f^{-1}gf=g^2\rangle$, and consider the system $\Lambda=\{y\cdot f=y, y\cdot g\neq y\}$.
Then the action $\tau:\mathrm{BS}(1,2)\to \textup{Homeo}^+(\R)$ given by $t\cdot \tau(f) =2t, t\cdot \tau(g)=t+1$ for $t\in \R$, and the point $0\in \R$ provides a solution. Indeed, one checks that $0\cdot \tau(f)=0$ and $0\cdot \tau(g)=1\neq 0$.

\begin{lemma}\label{Lem:System}
Given a countable group $H$, let $\Lambda_1,\ldots, \Lambda_m$ be finite systems over $H$.
If each $\Lambda_i$ is solvable over $H$, then $\Xi(\Lambda_1, \ldots ,\Lambda_m)$ is solvable over $H$.
\end{lemma}

\begin{proof}
We will show this for $m=2$, and the general case will follow from the same argument using induction, since $$\Xi(\Lambda_1,\ldots, \Lambda_m)=\Xi(\Xi(\Lambda_1,\ldots, \Lambda_{m-1}),\Lambda_m).$$
Consider actions $\sigma_i:H\to \textup{Homeo}^+(\R)$ and $x_i\in \R$ 
that are solutions to the systems $\Lambda_i$, for each $1\leq i\leq 2$.
We fix an enumeration of the orbit $O=x_1\cdot \sigma_1(H)$ as $(p_n)_{n\in \N}$, where $x_1=p_0$. We perform the ``blowup'' construction as follows. We replace each point $p_n$ by a closed interval $I_n$ of length 
$\frac{1}{2^n}$ to obtain a new copy of the real line, upon which the action $\sigma_1$ naturally extends to an action $\sigma_1^b:H\to \textup{Homeo}^+(\R)$ as follows. Whenever $p_i\cdot \sigma_1(\alpha)=p_j$ for some $\alpha\in H$, $\sigma_1^b(\alpha)$ maps the interval $I_i$ to $I_j$ 
by the unique orientation-preserving affine homeomorphism $\eta_{i,j}:I_i\to I_j$.
So the restriction $\sigma_1^b(\alpha)|_{I_i}=\eta_{i,j}$. 

The action $\sigma_1^b$ satisfies the following key property.
Recall that we fixed $x_1=p_0$.
For each equation $y\cdot \alpha_1=y$ and each inequation $y\cdot \alpha_2\neq y$ in $\Lambda_1$, it holds that:
\begin{enumerate}
\item $I_0$ is fixed pointwise by $\sigma_1^b(\alpha_1)$.
\item $(I_0\cdot \sigma_1^b(\alpha_2))\cap I_0=\emptyset$.
\end{enumerate}

Since each open interval $int(I_n)$ is homeomorphic to $\R$, we choose homeomorphisms $\nu_n:\R\to int(I_n)$ such that $\nu_i\eta_{i,j}=\nu_j$ for all $i,j$. 
Using topological conjugacy given by the maps $\nu_n:\R\to int(I_n)$, for each $n\in \N$ we build an action: 
$$\tau_{n}:H\to \textup{Homeo}^+(int(I_n)),\qquad \tau_n=\nu_n^{-1}\sigma_2\nu_n.$$  
We consider the diagonal-product action: 
$$\tau:H\to \textup{Homeo}^+(\R)\qquad \tau=\circledast_{n\in \N}\tau_n.$$
Note that for each pair $\alpha,\beta\in H$, by construction it holds that
$$\tau(\alpha)\sigma_1^b(\beta)=\sigma_1^b(\beta)\tau(\alpha).$$

This naturally induces an action: 
$$\lambda:H\oplus H\to \textup{Homeo}^+(\R),\qquad \lambda(\alpha,\beta)=\tau(\alpha)\sigma_1^b(\beta)=\sigma_1^b(\beta)\tau(\alpha).$$
By our hypothesis on $\sigma_2$, we have a point $z\in int(I_0)$ such that for each equation $y\cdot \alpha_1=y$ and inequation $y\cdot \alpha_2\neq y$ in $\Lambda_2$, it holds that 
$$z\cdot \tau(\alpha_1)=z,\qquad z\cdot \tau(\alpha_2)\neq z.$$
Let $\rho:H\to H\oplus H$ be the ``diagonal embedding'' given by $\rho(\alpha)=(\alpha,\alpha)$ for each $\alpha\in H$.
Then $\rho\lambda$ is the desired action and $z$ is the required point.
\end{proof}

\begin{remark}\label{InfiniteSystem}
Another proof of Lemma \ref{Lem:System} can be given as follows.
Let $I_1,\ldots , I_m$ be pairwise disjoint nonempty open intervals in $\R$.
We consider actions $\sigma_i:H\to \textup{Homeo}^+(I_i)$ and points $x_i\in I_i$ that witness the solution to $\Lambda_i$.
Take the diagonal product $\sigma=\circledast_{1\leq i\leq m}\sigma_i$.
Now consider the restriction of the action of $H$ on the orbit $(x_1,\ldots , x_m)\cdot \sigma(H)$.
Since each orbit $x_i\cdot \sigma_i(H)$ is totally ordered, we can lexicographically order $(x_1,\ldots , x_m)\cdot \sigma(H)$.
Clearly, this is a countable right $H$-invariant totally ordered set. By taking its dynamical realization in $\textup{Homeo}^+(\R)$, we obtain the desired action 
for which the image of $(x_1,\ldots , x_m)$ is the required point. The advantage of this proof is that it also works for any countably infinite set of systems.
\end{remark}

\begin{proof}[Proof of Theorem \ref{thm:StackingsRF} assuming Proposition \ref{Lem:Stackings}]

We enumerate the set of prefix pairs of $w$ as $\{(w_1^{(i)},w_2^{(i)})\mid 1\leq i\leq k\}$.
For each $1\leq i\leq k$, we define a system $\Lambda_i$ consisting of one equation and one inequation: 
$$y\cdot w=y,\qquad y\cdot w_1^{(i)}w_2^{(i)}(w_1^{(i)})^{-1} \neq y.$$
Applying Proposition \ref{Lem:Stackings}, for each $1\leq i\leq k$ we obtain a group action $\sigma_i:G\to \textup{Homeo}^+(\R)$ such that 
there exists some $x_i\in \R$ for which
$$x_i\cdot \sigma_i(w)=x_i,\qquad x_i\cdot \sigma_i(w_1^{(i)}w_2^{(i)})\neq x_i\cdot \sigma_i(w_1^{(i)}).$$
Applying Lemma \ref{Lem:System} to $\Lambda_1,\ldots ,\Lambda_k$, we obtain an action $\sigma:G\to \textup{Homeo}^+(\R)$ which witnesses a solution to the system $\Xi(\Lambda_1,\ldots ,\Lambda_k)$
for $G$ as some $x\in \R$. Note that each $\Lambda_i$ contains the same equation, $y\cdot w=y$, so we have that 
$$\Xi(\Lambda_1,\ldots ,\Lambda_k)=(\bigcup_{1\leq i\leq k}\Lambda_i^{\neq})\cup \{y\cdot w=y\}=\bigcup_{1\leq i\leq k}\Lambda_i.$$
The point $x\in \R$ that witnesses this for the action $\sigma$ satisfies that $\Omega(w,x)$ is stable.
This proves Theorem \ref{thm:StackingsRF}.
\end{proof}

We shall now focus on proving Proposition \ref{Lem:Stackings}.
First we simplify it to a proposition that is more convenient to prove. To do so, we make the following basic observation. It generalizes easily to any system over any group, but we only need this for systems consisting of one equation and one inequation. We will also use this in the proof of Proposition~\ref{Lem:StackingsSimple}.
\begin{lemma}\label{lemma: conjugation for systems}
    Fix a group $H$ and consider a system $\Lambda$ with one equation $y\cdot\alpha=y$ and one inequation $y\cdot\beta\neq y$, for some given $\alpha,\beta\in H$. For any $h\in H$, consider the \emph{$h$-conjugate system} $\Lambda'$ of $\Lambda$ with one equation $y\cdot (h^{-1}\alpha h)=y$ and one inequation $y\cdot (h^{-1}\beta h)\neq y$. Then $\Lambda$ has a solution over $H$ if and only if $\Lambda'$ does.
\end{lemma}
\begin{proof}
    Suppose $\Lambda$ has a solution $(\tau,x)$, that is, we have $x\cdot\tau(\alpha)=x$ and $x\cdot\tau(\beta)\neq x$.
    Then $(\tau, x\cdot \tau(h))$ is a solution for $\Lambda'$. The other direction works by symmetry.
\end{proof}

Now we show that Proposition \ref{Lem:Stackings} reduces to the following simpler statement.

\begin{proposition}\label{Lem:StackingsSimple}
Let $A,B$ be countable right-orderable groups and $G=A\star B$. Consider a cyclically reduced word
$$w=a_1b_1\ldots a_nb_n\in A\star B, \qquad a_i\in A\setminus \{id\}, b_i\in B\setminus \{id\},$$ 
and $w_1$ a proper prefix of $w$.
Assume that $w$ is not a proper power.
Then there exists an action $\tau:A\star B\to \textup{Homeo}^+(\R)$ and some $x\in \R$ such that $x\cdot \tau(w)=x$ and $x\cdot \tau(w_1)\neq x$.
\end{proposition}

\begin{proof}[Proof of Proposition \ref{Lem:Stackings} assuming Proposition \ref{Lem:StackingsSimple}]
The goal is to find a solution $(\sigma,x)$ for the following system $\Lambda$ over $G=A\star B$, where we have one equation $y\cdot w=y$ and one inequation $y\cdot (w_1w_2w_1^{-1})\neq y$.

By Lemma~\ref{lemma: conjugation for systems}, this is equivalent to solving the $w_1$-conjugate system $\Lambda'$ with one equation $y\cdot (w_1^{-1}ww_1)=y$ and one inequation $y\cdot w_2\neq y$.
Note that $w_1^{-1}ww_1$ is a cyclic conjugate of $w$ since $w_1$ is a prefix of $w$. And $w_2$ is a proper prefix of $w_1^{-1}ww_1$, so $\Lambda'$ has a solution by Proposition~\ref{Lem:StackingsSimple}.
\end{proof}

The goal of the rest of the section is to supply the proof of Proposition \ref{Lem:StackingsSimple}.
We will make use of the following elementary lemmas.

\begin{lemma}\label{lem: LOfact1}
Given a nontrivial element $f$ in a countable right-orderable group $H$, a nonempty open interval $I$, and $x,y\in I$ with $x<y$, we can find faithful actions $\nu_1,\nu_2:H\to \textup{Homeo}^+(I)$ such that $y<x\cdot \nu_1(f)$ and $y\cdot \nu_2(f)<x$.
\end{lemma}
\begin{proof}
Since $f\neq id$, for any $s\neq t\in I$, by conjugating a fixed faithful action, there is another faithful action $\nu:H\to \textup{Homeo}^+(I)$ such that $s\cdot \nu(f)=t$. The existence of $\nu_1$ and $\nu_2$ easily follows from this by choosing $s$ and $t$ appropriately.
\end{proof}

\begin{lemma}\label{lem: LOfact2}
Consider a countable right-orderable group $H$, a nonempty open interval $I$, and $x,y\in I$ with $x<y$. Let $f,g\in H\setminus \{id\}$ be such that $f\neq g$. Then we can find a faithful action $\sigma:H\to \textup{Homeo}^+(I)$ such that $y\cdot \sigma(g)<x\cdot \sigma(f)$.
\end{lemma}

\begin{proof}
We consider a dynamical realization $\sigma_1:H\to \textup{Homeo}^+(\R)$ of $H$ with a free orbit.
By conjugating with an orientation-reversing homeomorphism if needed, we can assume that there is a point $z\in \R$ such that $z\cdot \sigma_1(g)<z\cdot \sigma_1(f)$. Using continuity, we find points $x_1<y_1$ in a small neighborhood of $z$ such that $y_1\cdot \sigma_1(g)<x_1\cdot \sigma_1(f)$.
Now we consider a homeomorphism $\nu:\R\to I$ such that $\nu(x_1)=x, \nu(y_1)=y$. Upon conjugating $\sigma_1$ by $\nu$, we obtain the required $\sigma$.
\end{proof}

Before proving Proposition \ref{Lem:StackingsSimple}, first we discuss an elementary case as a warm-up.
\begin{example}\label{example: simplest case}
    Let $A,B$ be as above and let $w=ab, w_1=a$ for $a\in A\setminus \{id\}, b\in B\setminus \{id\}$. 
We would like to build an action $\eta:A\star B\to \textup{Homeo}^+(\R)$ which witnesses the conclusion of the Proposition~\ref{Lem:StackingsSimple} for these words, that is, $x\cdot\eta(ab)=x$ and $x\cdot\eta(a)\neq x$ for some $x\in\R$.

Let
$$I_1=[0,2], J_1=[1,3], K=[2,4], J_1'=[3,5], I_1'=[4,6].$$
Using Lemma \ref{lem: LOfact1}, it is straightforward to find actions
$$\sigma_1:A\to \textup{Homeo}^+(I_1),\qquad \tau_1:B\to \textup{Homeo}^+(J_1),$$ 
$$\sigma_1':A\to \textup{Homeo}^+(I_1'),\qquad  \tau_1':B\to \textup{Homeo}^+(J_1'),$$
such that 
$$1\cdot \sigma_1(a)>1.5,\qquad 1.5\cdot \tau_1(b)> 2.5,$$
$$5\cdot \sigma_1'(a)< 4.5,\qquad 4.5\cdot \tau_1'(b)<3.5.$$ 
Now let $\nu:A\to \textup{Homeo}^+(K)$ be an action such that $2.5\cdot \nu(a)> 3.5$.
Let
$$\sigma:A\to \textup{Homeo}^+[0,6],\qquad \sigma=\sigma_1\circledast \nu \circledast \sigma_1' ,$$ 
and 
$$\tau:B\to \textup{Homeo}^+[0,6],\qquad \tau=\tau_1\circledast \tau_1'.$$

We define $\eta=\langle \sigma,\tau\rangle$. We call such an action a \emph{dynamical arrangement}, which is formally defined below.
For 
$$K_1=\left[2.5,3.5\right ], \qquad K_2=[1,5],$$ 
our dynamical arrangement is designed to satisfy that
$$K_1\cdot \eta(ab)\subset K_2\cdot \eta(ab)\subset K_1,\qquad (K_1\cdot \eta(a))\cap K_1=\emptyset.$$ 
In particular, by the intermediate value theorem, $ab$ fixes a point in $K_1$ and $a$ moves every point in $K_1$.
\end{example}

Now we introduce some terminology for the general case. For $n\in \N$, a \emph{catenation of intervals} (see Figure~\ref{fig: catenation}) is an ordered tuple of the form $(I_1,J_1,\ldots ,I_n,J_n)$ or $(I_1,J_1,\ldots ,I_n,J_n,I_{n+1})$ of nonempty open intervals in $\R$ such that for all $k$ (in the given range):
\begin{enumerate}
\item Each interval $I_k,J_k$ has length $2$, and the endpoints lie in $\Z$;
\item $\sup(I_k)=\inf(I_{k+1})$ and $\sup(J_k)=\inf(J_{k+1})$; and
\item $\sup(I_k)$ is the midpoint of $J_k$ and $\inf(J_k)$ is the midpoint of $I_k$.
\end{enumerate}
Note that a catenation is an ordered tuple of intervals, and in some cases we shall use a different indexing system such as for instance, $(J_m,I_m,\ldots,J_1,I_1)$. However, it will still hold that the intervals appear ``from left to right'' starting from $J_m$ and onward until $I_1$.

Given a catenation of the form $(I_1,J_1,\ldots ,I_n,J_n)$ or $(I_1,J_1,\ldots ,I_n,J_n,I_{n+1})$, a \emph{dynamical arrangement} is an action of $A\star B$ which is generated by the diagonal-product of chosen actions of $A$ on the intervals $I_k$'s and the diagonal-product of chosen actions of $B$ on the intervals $J_k$'s.
We shall construct dynamical arrangements that allow us to prove Proposition \ref{Lem:StackingsSimple}.

We record two elementary lemmas concerning actions emerging from catenations.
These are essentially natural generalizations of a key idea in the example above.

\begin{lemma}\label{Lem:catenations1}
Let $A,B$ be countable right-orderable groups. Consider a reduced word 
$$w=a_1b_1\ldots a_kb_k\in A\star B,\qquad a_i\in A\setminus \{id\}, b_i\in B\setminus \{id\}.$$
Let $(I_1,J_1,\ldots ,I_k,J_k)$ be a catenation of intervals. Let $x=\inf(J_1)$, $y=\sup(J_k)$ and let $\epsilon>0$.
We can construct actions
$$\sigma_i:A\to \textup{Homeo}^+(I_i),\qquad \tau_i:B\to \textup{Homeo}^+(J_i),\qquad 1\leq i\leq k,$$
such that the following holds. Let $\nu=\langle \sigma,\tau\rangle$ be the action where $$\sigma=\circledast_{1\leq i\leq k}\sigma_i,\qquad \tau=\circledast_{1\leq i\leq k}\tau_i.$$
Then it holds that $x\cdot \nu(a_1b_1\ldots a_kb_k)\in (y-\epsilon,y)$.
Moreover, an analogous statement holds for words of the form $w=a_1b_1\ldots a_k$.
\end{lemma}

\begin{proof}
Fix $x=p_1$ for notational convenience.
First, we construct an action
$$\sigma_1:A\to \textup{Homeo}^+(I_1)\qquad \text{such that }q_1=p_1\cdot \sigma_1(a_1)\in int(J_1).$$
Next, we construct an action
$$\tau_1:B\to \textup{Homeo}^+(J_1)\qquad \text{such that } p_2=q_1\cdot\tau_1(b_1)\in int(I_2).$$
Proceeding inductively, we construct actions
$$\sigma_i:A\to \textup{Homeo}^+(I_i),\qquad \tau_i:B\to \textup{Homeo}^+(J_i),$$ 
such that:
$$q_i=p_i\cdot \sigma_i(a_i)\in int(J_i),\qquad \text{for } 1\leq i\leq k.$$ 
and 
$$p_{i+1}=q_i\cdot \tau_i(b_i)\in int(I_{i+1}),\qquad \text{for } 1\leq i<k.$$
Finally, we choose an action
$$\tau_k:B\to \textup{Homeo}^+(J_k)\qquad \text{such that }p_{k+1}\defeq q_k\cdot \tau_k(b_k)\in (y-\epsilon,y).$$
See Figure~\ref{fig: catenation}. Then the required actions are: $$\sigma=\circledast_{1\leq i\leq k}\sigma_i,\qquad \tau=\circledast_{1\leq i\leq k}\tau_i,\qquad \nu=\langle \sigma,\tau\rangle.$$
\end{proof}

\begin{figure}
	\labellist
	\small \hair 2pt
	\pinlabel $I_1$ at 30 35
    \pinlabel $J_1$ at 60 -5
    \pinlabel $I_2$ at 90 35
    \pinlabel $J_2$ at 120 -5

    \pinlabel $x=p_1$ at 17 23
    \pinlabel $q_1$ at 45 7
    \pinlabel $p_2$ at 75 23
    \pinlabel $q_2$ at 105 7
    \pinlabel $p_3$ at 135 23
    \pinlabel $\cdots$ at 195 25
    \pinlabel $\cdots$ at 225 5

    \pinlabel $I_k$ at 270 35
    \pinlabel $J_k$ at 300 -5

    \pinlabel $p_k$ at 255 23
    \pinlabel $q_k$ at 285 7
    \pinlabel $p_{k+1}$ at 330 23
    \pinlabel $y-\epsilon$ at 320 33
    \pinlabel $y$ at 335 7
    \endlabellist
	\centering
	\includegraphics[scale=1]{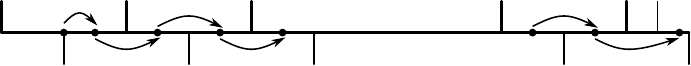}
	\caption{The action on a catenation in Lemma~\ref{Lem:catenations1}.}\label{fig: catenation}
\end{figure}
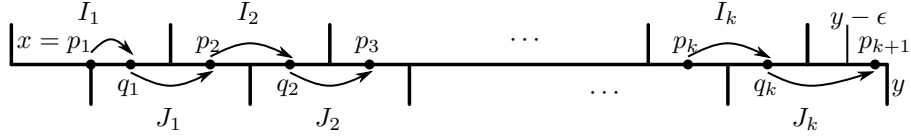

The following lemma is essentially a ``mirror image'' version of the previous lemma, whose proof is the same via conjugating by an orientation-reversing isometry of $\R$. 
\begin{lemma}\label{Lem:catenations2}
Let $A,B$ be countable right-orderable groups. Consider a reduced word 
$$w=a_1b_1\ldots a_kb_k\in A\star B,\qquad a_i\in A\setminus \{id\}, b_i\in B\setminus \{id\}.$$
Let $(J_k,I_k,\ldots ,J_1,I_1)$ be a catenation of intervals. Let $y=\inf(J_k)$, $x=\sup(J_1)$ and let $\epsilon>0$.
We can construct actions
$$\sigma_i:A\to \textup{Homeo}^+(I_i),\qquad \tau_i:B\to \textup{Homeo}^+(J_i),\qquad 1\leq i\leq k,$$
such that the following holds. 
Let $\nu=\langle \sigma,\tau\rangle$ be the action where
$$\sigma=\circledast_{1\leq i\leq k}\sigma_i,\qquad \tau=\circledast_{1\leq i\leq k}\tau_i.$$
Then it holds that $x\cdot \nu(a_1b_1\ldots a_kb_k)\in (y,y+\epsilon)$.
Moreover, an analogous statement holds for words of the form $w=a_1b_1\ldots a_k$.\qed
\end{lemma}



We are now ready to provide the final proof of this section.

\begin{proof}[Proof of Proposition \ref{Lem:StackingsSimple}]
The proof consists of two cases.
\begin{enumerate}
\item[{\bf Case $1$}:] $w_1=a_1b_1\ldots a_k$ for some $1\leq k\leq n$.
\item[{\bf Case $2$}:] $w_1=a_1b_1\ldots a_kb_k$ for some $1\leq k< n$.
\end{enumerate}

The proof strategy in both cases is to construct an action $\sigma:G\to \textup{Homeo}^+(\R)$ for which there is a closed interval
$I$ such that 
$$I\cdot \sigma(w)\subseteq I,\qquad (I\cdot \sigma(w_1))\cap I=\emptyset.$$
The former fact implies that $w$ admits a fixed point in $I$ due to the intermediate value theorem, and the latter fact implies that $w_1$ does not fix any point in $I$. Upon choosing $x$ to be the $w$-fixed point in $I$, this will finish the proof. In the second case, we will have to deal with more technicalities than the first, and that is the only instance where we use the fact that $w$ is not a proper power.

{\bf Proof of Case $1$}: 
The proof in this case is similar to that of Example~\ref{example: simplest case} considered above.
Consider a catenation (see Figure~\ref{fig: arrangement})
$$(I_1,J_1,\ldots ,I_n,J_n, L_1,M_1,\ldots, M_{k-1},L_k, J_n',I_n',\ldots,J_1',I_1').$$ 
We remark that the choice of indices is deliberate.
We let: 
$$x_1=\inf(J_1),\qquad x_2=\sup(J_1'),$$ 
$$y_1=\sup(J_n),\qquad y_2=\inf(J_n'),$$ 
$$z_1=y_1-\frac{1}{4},\qquad z_2=y_2+\frac{1}{4}.$$
See Figure~\ref{fig: arrangement}. Note that $z_1\in int(J_n\cap L_1)$, $z_2\in int(J_n'\cap L_k)$. 

Applying Lemmas~\ref{Lem:catenations1} and \ref{Lem:catenations2}, we choose faithful group actions
$$\nu_i:A\to \textup{Homeo}^+(I_i),\qquad \nu_i':A\to \textup{Homeo}^+(I_i'),\qquad 1\leq i\leq n,$$
$$\tau_i:B\to \textup{Homeo}^+(J_i),\qquad \tau_i':B\to \textup{Homeo}^+(J_i'),\qquad 1\leq i\leq n,$$
$$\delta_i:A\to \textup{Homeo}^+(L_i),\qquad 1\leq i\leq k,$$ $$\pi_i:B\to \textup{Homeo}^+(M_i),\qquad 1\leq i\leq k-1,$$
and build the action 
$$\sigma:G\to \textup{Homeo}^+(\R),\qquad \sigma=\langle \eta,\psi\rangle,$$ 
where 
$$\eta=(\circledast_{1\leq i\leq n}\nu_i)\circledast (\circledast_{1\leq i\leq k}\delta_i)\circledast (\circledast_{1\leq j\leq n}\nu_i'),$$
$$\psi=(\circledast_{1\leq i\leq n}\tau_i) \circledast (\circledast_{1\leq i\leq k-1}\pi_i)\circledast (\circledast_{1\leq j\leq n}\tau_i'),$$ 
and which satisfies  (see Figure~\ref{fig: arrangement})
\begin{enumerate}
\item $x_1\cdot \sigma(a_1b_1\ldots a_nb_n)\in (z_1,y_1)$.
\item $x_2\cdot \sigma(a_1b_1\ldots a_nb_n)\in (y_2,z_2)$.
\item $z_1\cdot \sigma(a_1b_1\ldots a_k)>z_2$.
\end{enumerate}

\begin{figure}
	\labellist
	\small \hair 2pt
	\pinlabel $I_1$ at 30 35
    \pinlabel $J_1$ at 60 -5
    \pinlabel $x_1$ at 22 10
    \pinlabel {$\sigma(a_1b_1\cdots a_n b_n)$} at 120 -7
    \pinlabel $J_n$ at 180 -5
    \pinlabel $z_1$ at 188 27
    \pinlabel $y_1$ at 218 10

    \pinlabel $L_1$ at 210 45
    \pinlabel {$\sigma(a_1b_1\cdots a_k)$} at 262 50
    \pinlabel $L_k$ at 315 45
    
    \pinlabel $I_1'$ at 525 35
    \pinlabel $J_1'$ at 495 -5
    \pinlabel $x_2$ at 534 10
    \pinlabel $J_n'$ at 345 -5
    \pinlabel {$\sigma(a_1b_1\cdots a_n b_n)$} at 420 -7
    
    \pinlabel $y_2$ at 307 10
    \pinlabel $z_2$ at 327 27

    \endlabellist
	\centering
	\includegraphics[scale=0.75]{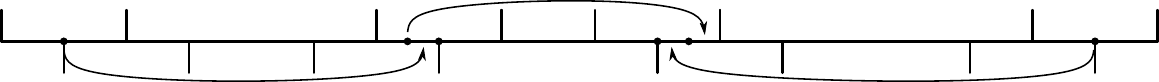}
	\caption{The arrangement for Case $1$ in the Proof of Proposition~\ref{Lem:StackingsSimple}.}\label{fig: arrangement}
\end{figure}
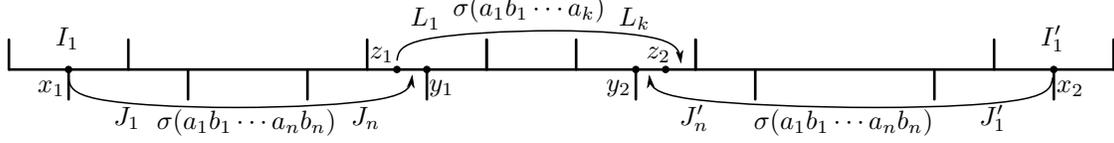

Let $I=[z_1,z_2]$.
We claim that the interval $I$ satisfies the required conditions.
From the first two parts and the fact that $x_1<z_1<y_1<y_2<z_2<x_2$, it follows that
$$I\cdot \sigma(a_1b_1\ldots a_nb_n)\subseteq I,$$
and hence $I$ contains a fixed point for the element $a_1b_1\ldots a_nb_n$.
However, the third part ensures that
$$(I\cdot \sigma(a_1b_1\ldots a_k))\cap I=\emptyset,$$
and hence no point in $I$ is fixed by $a_1b_1\ldots a_k$.
This settles the proof for this case.

{\bf Proof of Case $2$}: First we claim that we can assume without generality that $b_k\neq b_n$. Suppose $b_k=b_n$, then by Lemma~\ref{lemma: conjugation for systems}, it suffices to solve the $b_k^{-1}$-conjugate system, which has one equation $y\cdot (b_nwb_n^{-1})=y$ and one inequation $y\cdot (b_kw_1b_k^{-1})\neq y$. Note that
$$b_nwb_n^{-1}=b_na_1b_1\ldots a_n,\quad \text{ and}\quad b_kw_1b_k^{-1}=b_ka_1b_1\ldots a_k,$$
and the latter is again a proper prefix of the former and of even length. Next we check if $a_n=a_k$ and continue this process. 
If this never halts, then there is a subword $u$ whose length $|u|=\gcd(|w|,|w_1|)$ such that $w$ and $w_1$ are both powers of $u$, contradicting our assumption that $w$ is not a proper power.
Thus up to swapping the factors $A$ and $B$, we may assume without loss of generality that $b_k\neq b_n$.

From now on, assume $b_k\neq b_n$. Consider the catenation 
$$(I_1,J_1,\ldots ,I_n,J_n, L_1,M_1,\ldots, M_{k-1},L_k, J_n',I_n',\ldots,J_1',I_1').$$ 
We let
$$x_1=\inf(J_1),\qquad x_2=\sup(J_1'),\qquad y_1=\sup(J_n)-\frac{1}{2},\qquad y_2=\sup(J_n),$$ 
$$p_1=\inf(J_n'),\qquad p_2=\sup(J_n').$$ 
See Figure~\ref{fig: arrangement2}. Applying Lemmas \ref{Lem:catenations1} and \ref{Lem:catenations2}, we choose faithful groups actions:
$$\nu_i:A\to \textup{Homeo}^+(I_i),\qquad \nu_i':A\to \textup{Homeo}^+(I_i'),\qquad 1\leq i\leq n,$$
$$\tau_i:B\to \textup{Homeo}^+(J_i),\qquad 1\leq i\leq n,$$
$$\tau_i':B\to \textup{Homeo}^+(J_i'),\qquad 1\leq i<n.$$
We emphasize the occurrence of $i<n$ above. In particular, we will choose $\tau_n'$ at a later stage. Also choose the following actions using Lemmas \ref{Lem:catenations1} and \ref{Lem:catenations2}
$$\delta_i:A\to \textup{Homeo}^+(L_i),\qquad 1\leq i\leq k,$$ 
$$\pi_i:B\to \textup{Homeo}^+(M_i),\qquad 1\leq i\leq k-1,$$
and build the action 
$$\sigma':G\to \textup{Homeo}^+(\R)$$
given by $\sigma'=\langle \eta, \psi\rangle$, where
$$\eta=(\circledast_{1\leq i\leq n}\nu_i)\circledast (\circledast_{1\leq i\leq k}\delta_i)\circledast (\circledast_{1\leq i\leq n}\nu_i'),$$
$$\psi=(\circledast_{1\leq i\leq n}\tau_i)\circledast (\circledast_{1\leq i\leq k-1}\pi_i)\circledast (\circledast_{1\leq i< n}\tau_i'),$$ 
and which satisfies (see Figure~\ref{fig: arrangement2}):
$$x_1\cdot \sigma'(a_1b_1\ldots a_nb_n)\in (y_1,y_2),$$ 
$$u_1\defeq y_1\cdot \sigma'(a_1b_1\ldots a_k)\in (p_1,p_2),$$
$$u_2\defeq x_2\cdot \sigma'(a_1b_1\ldots a_n)\in (p_1,p_2).$$
Note that $u_1<u_2$, since the midpoint of $J_n'$ lies in $(u_1,u_2)$.
Using Lemma \ref{lem: LOfact2}, we construct an action $\tau_n':B\to \textup{Homeo}^+(J_n')$ such that (see the top of Figure~\ref{fig: arrangement2}): $$v_2=u_2\cdot \tau_n'(b_n)<u_1\cdot \tau_n'(b_k)=v_1$$

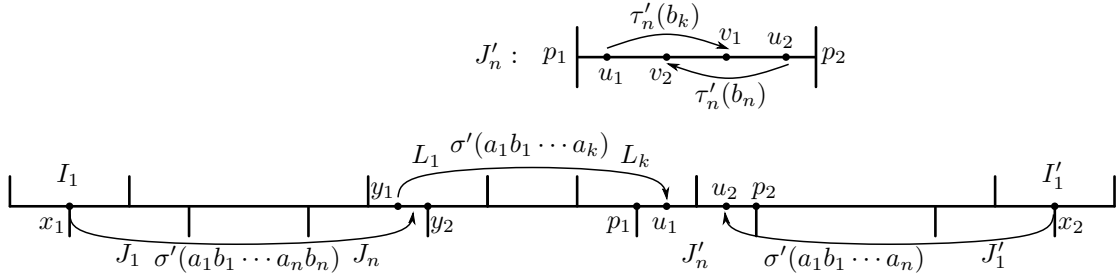
\begin{figure}
	\labellist
	\small \hair 2pt
	\pinlabel $I_1$ at 30 35
    \pinlabel $J_1$ at 60 -5
    \pinlabel $x_1$ at 22 10
    \pinlabel {$\sigma'(a_1b_1\cdots a_n b_n)$} at 120 -7
    \pinlabel $J_n$ at 180 -5
    \pinlabel $y_1$ at 188 27
    \pinlabel $y_2$ at 218 10

    \pinlabel $L_1$ at 210 45
    \pinlabel {$\sigma'(a_1b_1\cdots a_k)$} at 262 50
    \pinlabel $L_k$ at 315 45
    
    \pinlabel $I_1'$ at 525 35
    \pinlabel $J_1'$ at 495 -5
    \pinlabel $x_2$ at 534 10
    \pinlabel $J_n'$ at 345 -5
    \pinlabel {$\sigma'(a_1b_1\cdots a_n)$} at 420 -7
    
    \pinlabel $p_1$ at 307 10
    \pinlabel $u_1$ at 330 10
    \pinlabel $u_2$ at 360 27
    \pinlabel $p_2$ at 380 27

    \pinlabel $J'_n:$ at 245 95
    \pinlabel $p_1$ at 275 95
    \pinlabel $u_1$ at 303 85
    \pinlabel $p_2$ at 415 95
    \pinlabel $u_2$ at 388 103
    \pinlabel $v_2$ at 328 85
    \pinlabel $v_1$ at 363 105
    \pinlabel $\tau'_n(b_n)$ at 363 75
    \pinlabel $\tau'_n(b_k)$ at 330 115
    \endlabellist
	\centering
	\includegraphics[scale=0.75]{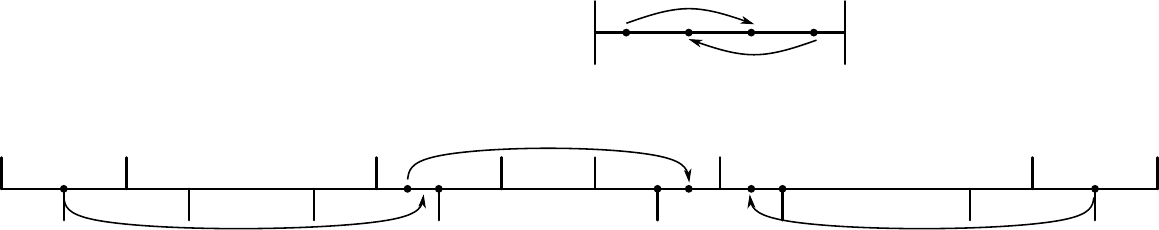}
	\caption{The arrangement for Case $2$ in the Proof of Proposition~\ref{Lem:StackingsSimple}, where the crucial interval $J'_n$ has been magnified. The positions of $v_1,v_2$ relative to $u_1,u_2$ are irrelevant, and here we show one possibility.}\label{fig: arrangement2}
\end{figure}

Since the action $\psi$ pointwise fixes $J_n'$ by definition, we can define $\psi'=\psi\circledast \tau_n'$.
We then define $\sigma=\langle \eta, \psi'\rangle$.
We claim that our construction satisfies the requirements for $I=[y_1,v_2]$.
To see this, note that 
$$x_2\cdot\sigma(a_1b_1\ldots a_nb_n)=u_2\cdot\sigma(b_n)=u_2\cdot\tau'_n(b_n)=v_2,$$
so
$$(x_1,x_2)\cdot \sigma(a_1b_1\ldots a_nb_n)\subset (y_1,v_2).$$
Since $I=[y_1,v_2]\subset (x_1,x_2)$, we deduce that
$$I\cdot \sigma(a_1b_1\ldots a_nb_n)\subseteq I.$$
Moreover, $$y_1\cdot \sigma(a_1b_1\ldots a_kb_k)=u_1\cdot \sigma(b_k)=u_1\cdot \tau_n'(b_k)=v_1> v_2.$$
Therefore, it follows that 
$$(I\cdot \sigma(a_1b_1\ldots a_kb_k))\cap I=\emptyset.$$
So the interval $I=[y_1,v_2]$ contains a $\sigma(a_1b_1\ldots a_nb_n)$-fixed point, but no point in $I$ is fixed by $\sigma(a_1b_1\ldots a_kb_k)$.
This settles the proof for this case.
\end{proof}

\begin{remark}
Let $A,B,w$ be as in the statement of Theorem \ref{thm:StackingsRF} (so $w$ is not a proper power).
Using the methods in our section, we can actually show something much stronger. The right action of $G=A\star B$ on the set of right cosets of $H=\langle w\rangle$ in $G$ admits a $G$-invariant order.
This means that the group $H$ is \emph{right relatively convex} in $G$. Equivalently, $G$ admits an action $\sigma:G\to \textup{Homeo}^+(\R)$ and a point $x\in \R$ such that the stabilizer of 
$x$ in $G$ is precisely $H$. Let $\{\alpha_n\mid n\in \N\}$ be an enumeration of a system of coset representatives of the nontrivial right cosets of $H$ in $G$.
Then using an argument similar to the one above, for each $i\in \N$ we can find an action $\sigma_i:G\to \textup{Homeo}^+(\R)$ and a point $x_i\in \R$ such that $x_i\cdot \sigma(w)=x_i$ and $x_i\cdot \sigma_i(\alpha_i)\neq x_i$ (we call this system $\Lambda_i$). Then applying the argument presented in Remark \ref{InfiniteSystem} which finds a simultaneous solution for the countably infinite set of systems $\{\Lambda_i\mid i\in\N\}$, we find the required action.
\end{remark}

\section{Proof of the main results}\label{sec: proof of main results}
Now we have proved the two main ingredients, Theorems~\ref{thmA: spectral gap from stacking} and \ref{thmA: Stacking}. Using them, we first deduce Theorem~\ref{thmA: spectral gap theorem}.
\begin{proof}[Proof of Theorem~\ref{thmA: spectral gap theorem}]
    Here $S$ has no $2$-sphere or disk components, so $-\chi^-(S)=-\chi(S)$.
    
    First consider the case where $w$ is not a proper power.
    All notions involved are preserved by conjugation, and since $w$ is not conjugate into $A$ or $B$, we may assume that $w$ is expressed as a cyclically reduced word. Then by Theorem~\ref{thmA: Stacking}, there is a relative stacking for $w$, and hence by Theorem~\ref{thmA: spectral gap from stacking} we have $-\chi^-(S)\ge \deg(S)$. If $A$ or $B$ are uncountable, pass to the countable subgroups $A'$ and $B'$ containing the letters in $w$, then they are right-orderable since they are subgroups of $A$ and $B$ respectively, so there is a relative stacking for $w\in G'=A'\star B'$, and the result follows using Theorem~\ref{thm: spectral gap from stacking} instead of Theorem~\ref{thmA: spectral gap from stacking}.
    
    Now suppose $w=u^k$ is a proper power, where $u\in G$ is not a proper power and $k\ge2$. Note that $u$ is not conjugate into $A$ or $B$ since $w$ is not, so the conclusion holds for $u$. 
    Any $w$-admissible surface $S$ of degree $d=\deg(S)$ is naturally also a $u$-admissible surface of degree $kd$, and the notion of boundary-incompressibility is preserved.
    From what we have shown for $u$, we obtain $-\chi^-(S)\ge kd=k\deg(S)\ge \deg(S)$.
\end{proof}

Finally, we deduce Theorem~\ref{thmA: main} from Theorem~\ref{thmA: spectral gap theorem}.
\begin{proof}[Proof of Theorem~\ref{thmA: main}]
    We first show the second assertion, namely, for any $w\in G$ not conjugate into $A$, the natural map $A\inj A\star B$ induces an \emph{injection} $A\inj (A\star B)/\llangle w \rrangle$. 
    This is obvious if $w$ is conjugate to some $b\in B$ since $(A\star B)/\llangle w \rrangle= A\star (B/\llangle b\rrangle)$. So we assume $w$ is not conjugate into $B$ either. 
    
    Suppose the map is not injective, namely, there is an $a\in A\setminus \{id\}$ such that $a\in \llangle w\rrangle$. Then by Example~\ref{example: equations in groups}, there are equations of the form (\ref{eqn: group equation}), and we consider one with minimal length $k\ge1$. As explained in Example~\ref{example: equations in groups}, this gives rise to a boundary-incompressible $w$-admissible surface $S$ which is a sphere with $k+1$ boundary components and of degree 
    $$\deg(S)=\sum_{i=1}^k |n_i|\ge k.$$
    Note that $-\chi^-(S)=-\chi(S)=(k+1)-2=k-1$.
    By Theorem~\ref{thmA: spectral gap theorem}, we get
    $$k-1=-\chi^-(S)\ge \deg(S)\ge k,$$
    which is a contradiction.

    This implies that $w$ does not normally generate $A\star B$ when $w$ is not conjugate into $A$. If $w$ is conjugate to some $a\in A$, then $(A\star B)/\llangle w \rrangle= (A/\llangle a\rrangle)\star B$, which is nontrivial. Hence $A\star B$ has normal rank greater than $1$ for any $A$ and $B$ right-orderable and nontrivial.
\end{proof}

With essentially the same proof, we deduce the following:
\begin{corollary}\label{cor: torsion}
    If $A$ and $B$ are right-orderable, and $u\in G=A\star B$ is not conjugate into $A$ or $B$ and is not a proper power, then for any $k\ge2$, the image of $u$ in $G/\llangle u^k \rrangle$ has order $k$.
\end{corollary}
\begin{proof}
    Denote by $\bar{u}$ the image of $u$ in  $G/\llangle u^k \rrangle$. Clearly $\bar{u}^k=id$. It remains to show $u^m\notin \llangle u^k \rrangle$ whenever $m$ is not divisible by $k$.
    
    Suppose there is such $m$ with $u^m\in \llangle u^k \rrangle$, then we can get an equation
    $$
        u^m= (g_1 u^{kn_1} g_1^{-1})(g_2 u^{kn_2} g_2^{-1})\cdots (g_\ell u^{kn_\ell} g_\ell^{-1}),
    $$
    where $\ell\in \Z_+$, each $g_i\in G$ and $n_i\neq0\in\Z$ for all $1\le i\le \ell$.
    This yields a surface map $f:S\to X$ for $X=K(G,1)$ such that $S$ is a sphere with $\ell+1$ boundary components representing $u^m, u^{kn_1}, \ldots, u^{kn_\ell}$ respectively. This is a $u$-admissible surface of degree 
    $$\deg(S)=|m|+k\sum_{i=1}^\ell|n_i|\ge k\ell\ge \ell.$$

    Among all choices of $m$ not divisible by $k$ and all equations of the form above, choose one with minimal length $\ell\ge1$. Then the corresponding $S$ is boundary-incompressible by minimality. Thus by Theorem~\ref{thmA: spectral gap theorem}, we have
    $$\ell-1=-\chi(S)\ge \deg(S)\ge \ell,$$
    which gives a contradiction.
\end{proof}

Under the stronger assumption that $A$ and $B$ are locally indicable, Howie \cite{Howie_LocInd} showed that $A\star B/\llangle w\rrangle$ is locally indicable if and only if it is torsion-free and if and only if $w$ is not a proper power. It is natural to ask for an analog, replacing local indicability by right-orderability, and Corollary~\ref{cor: torsion} gives a partial implication.
\begin{question}
    If $A$ and $B$ are right-orderable, and $w\in G=A\star B$ is not conjugate into $A$ or $B$ and is not a proper power. Is $G/\llangle w\rrangle$ torsion-free? Is it right-orderable?
\end{question}

\section{Applications to Dehn surgery}\label{sec: Dehn surgery}
Now we give the detailed proof of Theorem~\ref{thmA: Dehn surgery}. The construction of such manifolds was pointed out to us by Nathan Dunfield, who also explained the fact that they admit degree-one maps to a reduced manifold of the form in Corollary~\ref{cor: connected sum}.

\begin{construction}\label{construction}
    Let $M=\#_{i=1}^k M_i$ be the connected sum of $k\ge2$ closed oriented $3$-manifolds $M_i$ so that each $\pi_1(M_i)$ is nontrivial and left-orderable. Fix any null-homotopic hyperbolic knot $K$ in $M$, which exists by \cite[Prop 4.2]{BoileauWang}. 
    Let $M_K(\gamma)$ be the result of applying a Dehn surgery with slope $\gamma$ for some simple loop $\gamma$ on the boundary of a solid torus neighborhood $N(K)$ of $K$.
\end{construction}

There are many manifolds $M_i$ meeting the requirement in the construction above, including infinitely many $\Z$-homology spheres, say by \cite[Proposition 1(2)]{LSpace}.

The Dehn surgery $M_K(\gamma)$ only depends on $[\gamma]\in H_1(T;\Z)$, where $T$ is the boundary torus of $N(K)$. Usually one chooses a basis $[\mu],[\lambda]$ for $H_1(T;\Z)$ represented by the meridian $\mu$, which is the unique loop up to homotopy that bounds a disk in the solid torus $N(K)$, and by the longitude $\lambda$, which is some loop such that $[\mu]\cap[\lambda]=1$. There is a $\Z$-worth of choices for $\lambda$, but when $M$ is a $\Z$-homology sphere, there is a unique choice of $\lambda$ so that it is null-homologous in $M\setminus N(K)$. So when $M$ is a $\Z$-homology sphere, we write $M_K(p/q)=M_K(\lambda)$ if
$[\gamma]=p[\mu]+q[\lambda]$ for some coprime integers $p,q\in\Z$, where $p/q\in\Q\cup\{\infty\}$. Then $M_K(1/n)$ is again a $\Z$-homology sphere for any $n\in\Z$. We record this together with another well-known fact.

\begin{lemma}\label{lemma: surgery facts}
    If each $M_i$ is a $\Z$-homology sphere in Construction \ref{construction}, then $M_K(1/n)$ is a $\Z$-homology sphere for any $n\in\Z$. Moreover, if $K$ is a hyperbolic knot in $M$, then for $n$ sufficiently large, $M_K(1/n)$ is hyperbolic and there are infinitely different homeomorphism classes as $n$ varies.
\end{lemma}
\begin{proof}
    We have explained the first fact above. 
    The second fact follows from Thurston's hyperbolic Dehn surgery theorem \cite{Thurstonnotes}: All but finitely many choices of the slopes $p/q$ yield hyperbolic manifolds $M_K(p/q)$, and the hyperbolic volume converges to the hyperbolic volume of $M\setminus N(K)$ as $p^2+q^2\to\infty$.
    The hyperbolic volume of $M_K(p/q)$ is a topological invariant by the Mostow rigidity.
\end{proof}

\begin{lemma}\label{lemma: M' normal rank}
    For each $M'\defeq M_K(\gamma)$ in Construction \ref{construction}, there is a degree-one map $f:M'\to M$. In particular, the normal rank of $\pi_1(M')$ is at least $2$.
\end{lemma}
\begin{proof}
    Any degree-one map of closed orientable manifolds induces a surjection of fundamental groups, so the normal rank of $\pi_1 M'$ is no less than that of $\pi_1 M$, which is at least $2$ by Corollary~\ref{cor: connected sum}. 
    
    The existence of a degree-one map is due to \cite[Prop 3.2]{BoileauWang}, which we restate below removing an unnecessary assumption on the irreducibility of $M$.
\end{proof}

\begin{proposition}[{Boileau--Wang \cite[Prop 3.2]{BoileauWang}}]
    Let $K$ be a null-homotopic knot in a closed orientable connected manifold $M$, and let $M'=M_K(\gamma)$ be the result of a Dehn surgery along $K$. Then there is a degree-one map $f:M'\to M$.
\end{proposition}
\begin{proof}
    We sketch their construction of the degree-one map to explain why irreducibility of $M$ is unnecessary. This is literally a summary of their original proof.

    The fact that $K$ is null-homotopic implies that it bounds a (singular) disk $\varphi: D^2\to M$ whose image $D=\varphi(D^2)$ has the homotopy type of a graph, and so does a small neighborhood $N(D)$ of $D$. It follows that for any loop $\gamma$ on a small solid torus neighborhood $N(K)\subset N(D)$, there is a null-homotopy of $K$ supported inside $N(D)$.

    Remove from $M$ the small open neighborhood $N(K)\subset N(D)$ of $K$ and denote the result as $\widehat{M}$. Then $M'=\widehat{M}\cup (S^1\times D^2)$. Define $f|_{\widehat{M}}=id_{\widehat{M}}$, so once we extend it to a continuous map $f:M'\to M$ it must have degree one.

    To extend $f$ over $S^1\times D^2$, say $\{1\}\times D^2$ bounds the surgery slope $\gamma$, we first extend $f$ over $\{1\}\times D^2$ as a null-homotopy supported inside $N(D)$. Then we extend to the rest of the solid torus using the fact that $N(D)$ has trivial $\pi_2$, since it has the homotopy type of a graph.
\end{proof}

In particular, by Lemma~\ref{lemma: M' normal rank}, any manifold $M_K(\gamma)$ from Construction~\ref{construction} cannot be obtained by a Dehn surgery along a knot in $S^3$. Actually, as Tye Lidman explained to us, one can replace $S^3$ by $S^2\times S^1$ or any lens space using the argument below.
\begin{proposition}\label{prop: cobordism}
    If $Y^3$ is a $\Z$-homology sphere such that $\pi_1 Y$ has normal rank at least $2$, then $Y$ cannot be obtained by a Dehn surgery along any knot in $S^3$, $S^2\times S^1$ or any lens space.
\end{proposition}
\begin{proof}
    Suppose $Y$ can be obtained by a Dehn surgery along a knot $K$ in $X$, where $X$ is a $3$-manifold with abelian fundamental group. This includes the cases of $X$ being $S^3$, $S^2\times S^1$ and lens spaces. Let $Z$ be the complement of an open tubular neighborhood of $K$, and denote its torus boundary by $T=\partial Z$. Then this means that there are two Dehn fillings along loops $\alpha,\beta\subset T$ that result in $X$ and $Y$ respectively. So 
    $$\pi_1 X=\pi_1 Z/\llangle \alpha\rrangle\quad \text{and}\quad \pi_1 Y=\pi_1 Z/\llangle \beta\rrangle.$$
    Let $\gamma$ be a loop in $T$ such that $\{\beta,\gamma\}$ form a basis of $\pi_1 T\cong\Z^2$.
    Then 
    $$\pi_1 Z/\llangle\pi_1T\rrangle=\pi_1 Z/\llangle \beta,\gamma \rrangle=\pi_1 Y/\llangle\bar{\gamma}\rrangle,$$
    where $\bar{\gamma}$ is the image of $\gamma$ in $\pi_1 Y$.
    In particular, $\pi_1 Z/\llangle\pi_1T\rrangle$ must be perfect, since it is a quotient of a perfect group $\pi_1 Y$.
    On the other hand, $\pi_1 Z/\llangle\pi_1T\rrangle$ is also a quotient of $\pi_1 X=\pi_1 Z/\llangle \alpha\rrangle$, so it must be abelian.
    A perfect abelian group must be trivial, so 
    $$1=\pi_1 Z/\llangle\pi_1T\rrangle=\pi_1 Y/\llangle \bar{\gamma}\rrangle,$$
    contradicting the fact that $\pi_1 Y$ has normal rank at least $2$.
\end{proof}

Combining the results above we obtain the following theorem, which is a more precise version of Theorem~\ref{thmA: Dehn surgery}.
\begin{theorem}
    Given $M$ and $K$ as in Construction~\ref{construction}, where we choose each $M_i$ to be a $\Z$-homology sphere, then $M'=M_K(1/n)$ is a hyperbolic $\Z$-homology sphere for all $n$ sufficiently large, and $\pi_1(M')$ has normal rank at least $2$. 
    
    As a consequence, each such $M'$ cannot be obtained by a Dehn surgery along any knot in $S^3$, $S^2\times S^1$ or any lens space. Moreover, this gives infinitely many homeomorphism classes by varying $n$ in the family $M_K(1/n)$.
\end{theorem}
\begin{proof}
    By Lemma~\ref{lemma: surgery facts}, $M'=M_K(1/n)$ is a hyperbolic $\Z$-homology sphere for all $n$ sufficiently large, and we get infinitely many homeomorphism classes as we vary $n$. Each such $M'$ has normal rank at least $2$ by Lemma~\ref{lemma: M' normal rank}. The assertion about Dehn surgery follows from Proposition~\ref{prop: cobordism}.
\end{proof}

\section*{Acknowledgments}
The authors are grateful to Henry Wilton for suggesting us to think about a relative version of the stacking argument in \cite{LouderWilton:stacking}. 
The authors are also in debt to Doron Puder for suggesting and explaining the stacking argument of Louder--Wilton and pointing out the (different) generalization of stacking in \cite{Millard} in some helpful discussions that motivated the study which eventually led us to the argument here.
The authors are grateful to Francesco Fournier-Facio for a careful reading of an early draft of the paper, and for providing valuable feedback.
The authors also thank Denis Osin for helpful discussions.
The authors are also very grateful to Ian Agol, Nathan Dunfield, Tye Lidman, Yi Ni, and Hongbin Sun for explaining applications of our results to $3$-manifold topology.
LC is partially supported by NSF DMS-2506702. YL is partially supported by NSF DMS-2552707.

\bibliographystyle{alpha}
\bibliography{wiegold}

\begin{thebibliography}{PSEWS25}

\bibitem[Ago13]{Agol}
Ian Agol.
\newblock The virtual {H}aken conjecture.
\newblock {\em Doc. Math.}, 18:1045--1087, 2013.
\newblock With an appendix by Agol, Daniel Groves, and Jason Manning.

\bibitem[Auc97a]{surg1}
David Auckly.
\newblock Surgery numbers of {$3$}-manifolds: a hyperbolic example.
\newblock In {\em Geometric topology ({A}thens, {GA}, 1993)}, volume 2.1 of
  {\em AMS/IP Stud. Adv. Math.}, pages 21--34. Amer. Math. Soc., Providence,
  RI, 1997.

\bibitem[Auc97b]{Auckly}
David Auckly.
\newblock Surgery numbers of {$3$}-manifolds: a hyperbolic example.
\newblock In {\em Geometric topology ({A}thens, {GA}, 1993)}, volume 2.1 of
  {\em AMS/IP Stud. Adv. Math.}, pages 21--34. Amer. Math. Soc., Providence,
  RI, 1997.

\bibitem[BGW13]{LSpace}
Steven Boyer, Cameron~McA. Gordon, and Liam Watson.
\newblock On {L}-spaces and left-orderable fundamental groups.
\newblock {\em Math. Ann.}, 356(4):1213--1245, 2013.

\bibitem[BH72]{BurnsHale}
R.~G. Burns and V.~W.~D. Hale.
\newblock A note on group rings of certain torsion-free groups.
\newblock {\em Canad. Math. Bull.}, 15:441--445, 1972.

\bibitem[BMS02]{MR1921705}
Gilbert Baumslag, Alexei~G. Myasnikov, and Vladimir Shpilrain.
\newblock Open problems in combinatorial group theory. {S}econd edition.
\newblock In {\em Combinatorial and geometric group theory ({N}ew {Y}ork,
  2000/{H}oboken, {NJ}, 2001)}, volume 296 of {\em Contemp. Math.}, pages
  1--38. Amer. Math. Soc., Providence, RI, 2002.

\bibitem[Bro84]{Brodskii}
S.~D. Brodski\u{\i}.
\newblock Equations over groups, and groups with one defining relation.
\newblock {\em Sibirsk. Mat. Zh.}, 25(2):84--103, 1984.

\bibitem[BW96]{BoileauWang}
Michel Boileau and Shicheng Wang.
\newblock Non-zero degree maps and surface bundles over {$S^1$}.
\newblock {\em J. Differential Geom.}, 43(4):789--806, 1996.

\bibitem[Cal09]{Cal:rational}
Danny Calegari.
\newblock Stable commutator length is rational in free groups.
\newblock {\em J. Amer. Math. Soc.}, 22(4):941--961, 2009.

\bibitem[CD18]{CullerDunfield}
Marc Culler and Nathan~M. Dunfield.
\newblock Orderability and {D}ehn filling.
\newblock {\em Geom. Topol.}, 22(3):1405--1457, 2018.

\bibitem[CH19]{CH:sclgap}
Lvzhou Chen and Nicolaus Heuer.
\newblock Spectral gap of scl in graphs of groups and $3$-manifolds, {arXiv}
  1910.14146, 2019.

\bibitem[Che18]{Chen:sclfpgap}
Lvzhou Chen.
\newblock Spectral gap of scl in free products.
\newblock {\em Proc. Amer. Math. Soc.}, 146(7):3143--3151, 2018.

\bibitem[Che20]{Chen:sclBS}
Lvzhou Chen.
\newblock Scl in graphs of groups.
\newblock {\em Invent. Math.}, 221(2):329--396, 2020.

\bibitem[Che25]{Chen:Kervaire}
Lvzhou Chen.
\newblock The kervaire conjecture and the minimal complexity of surfaces,
  {arXiv} 2302.09811, 2025.

\bibitem[Deh00]{Dehornoy}
Patrick Dehornoy.
\newblock {\em Braids and self-distributivity}, volume 192 of {\em Progress in
  Mathematics}.
\newblock Birkh\"{a}user Verlag, Basel, 2000.

\bibitem[DH91]{DH91}
Andrew~J. Duncan and James Howie.
\newblock The genus problem for one-relator products of locally indicable
  groups.
\newblock {\em Math. Z.}, 208(2):225--237, 1991.

\bibitem[DNR14]{GOD}
B.~Deroin, A.~Navas, and C.~Rivas.
\newblock Groups, orders, and dynamics.
\newblock {\em arXiv:1408.5805}, 2014.

\bibitem[EM13]{MonodEisenmann}
A.~Eisenmann and N.~Monod.
\newblock Normal generation of locally compact groups.
\newblock {\em Bull. Lond. Math. Soc.}, 45(4):734--738, 2013.

\bibitem[GAn75]{Gonzalez}
F.~Gonz\'alez-Acu\~na.
\newblock Homomorphs of knot groups.
\newblock {\em Ann. of Math. (2)}, 102(2):373--377, 1975.

\bibitem[GAnS86]{AcunaShort}
Francisco Gonz\'{a}lez-Acu\~{n}a and Hamish Short.
\newblock Knot surgery and primeness.
\newblock {\em Math. Proc. Cambridge Philos. Soc.}, 99(1):89--102, 1986.

\bibitem[GL89]{GordonLuecke}
C.~McA. Gordon and J.~Luecke.
\newblock Knots are determined by their complements.
\newblock {\em Bull. Amer. Math. Soc. (N.S.)}, 20(1):83--87, 1989.

\bibitem[Gor83]{Gordon:DehnSurg}
C.~McA. Gordon.
\newblock Dehn surgery and satellite knots.
\newblock {\em Trans. Amer. Math. Soc.}, 275(2):687--708, 1983.

\bibitem[GR62]{GerstenhaberRothaus}
Murray Gerstenhaber and Oscar~S. Rothaus.
\newblock The solution of sets of equations in groups.
\newblock {\em Proc. Nat. Acad. Sci. U.S.A.}, 48:1531--1533, 1962.

\bibitem[GS87]{GhysSergiescu}
\'{E}tienne Ghys and Vlad Sergiescu.
\newblock Sur un groupe remarquable de diff\'{e}omorphismes du cercle.
\newblock {\em Comment. Math. Helv.}, 62(2):185--239, 1987.

\bibitem[Heu19]{Heuer}
Nicolaus Heuer.
\newblock Gaps in {SCL} for amalgamated free products and {RAAG}s.
\newblock {\em Geom. Funct. Anal.}, 29(1):198--237, 2019.

\bibitem[HKL16]{surg2}
Jennifer Hom, {\c C}a{\u g}r{\i{}} Karakurt, and Tye Lidman.
\newblock Surgery obstructions and {H}eegaard {F}loer homology.
\newblock {\em Geom. Topol.}, 20(4):2219--2251, 2016.

\bibitem[HL18]{HomLidman}
Jennifer Hom and Tye Lidman.
\newblock A note on surgery obstructions and hyperbolic integer homology
  spheres.
\newblock {\em Proc. Amer. Math. Soc.}, 146(3):1363--1365, 2018.

\bibitem[HL19]{HydeLodha}
James Hyde and Yash Lodha.
\newblock Finitely generated infinite simple groups of homeomorphisms of the
  real line.
\newblock {\em Invent. Math.}, 218(1):83--112, 2019.

\bibitem[HL25]{HydeLodhaFP}
James Hyde and Yash Lodha.
\newblock Finitely presented simple left-orderable groups in the landscape of
  {R}ichard {T}hompson's groups.
\newblock {\em Ann. Sci. \'{E}c. Norm. Sup\'{e}r. (4)}, 58(2):419--432, 2025.

\bibitem[How81]{Howie_LocIndFrei}
James Howie.
\newblock On pairs of {$2$}-complexes and systems of equations over groups.
\newblock {\em J. Reine Angew. Math.}, 324:165--174, 1981.

\bibitem[How82]{Howie_LocInd}
James Howie.
\newblock On locally indicable groups.
\newblock {\em Math. Z.}, 180(4):445--461, 1982.

\bibitem[How02]{Howie_cyclic}
James Howie.
\newblock A proof of the {S}cott-{W}iegold conjecture on free products of
  cyclic groups.
\newblock {\em J. Pure Appl. Algebra}, 173(2):167--176, 2002.

\bibitem[Joh80]{Johnson}
Dennis Johnson.
\newblock Homomorphs of knot groups.
\newblock {\em Proc. Amer. Math. Soc.}, 78(1):135--138, 1980.

\bibitem[Kir78]{Kirby:oldlist}
Rob Kirby.
\newblock Problems in low dimensional manifold theory.
\newblock In {\em Algebraic and geometric topology ({P}roc. {S}ympos. {P}ure
  {M}ath., {S}tanford {U}niv., {S}tanford, {C}alif., 1976), {P}art 2}, Proc.
  Sympos. Pure Math., XXXII, pages 273--312. Amer. Math. Soc., Providence,
  R.I., 1978.
\newblock Available at
  \url{https://math.berkeley.edu/~kirby/papers/Kirby%20-%20Problems%20in%20low%20dimensional%20manifold%20theory%20-%20MR0520548.pdf}.

\bibitem[Kir97]{Kirby95}
Problems in low-dimensional topology.
\newblock In Rob Kirby, editor, {\em Geometric topology ({A}thens, {GA},
  1993)}, volume 2.2 of {\em AMS/IP Stud. Adv. Math.}, pages 35--473. Amer.
  Math. Soc., Providence, RI, 1997.

\bibitem[Kly93]{Klyachko}
Anton~A. Klyachko.
\newblock A funny property of sphere and equations over groups.
\newblock {\em Comm. Algebra}, 21(7):2555--2575, 1993.

\bibitem[KM23]{Kourovka}
E.~I. Khukhro and V.~D. Mazurov.
\newblock Unsolved problems in group theory. {The Kourovka Notebook}, {arXiv}
  1401.0300, 2023.

\bibitem[Lev62]{Levin}
Frank Levin.
\newblock Solutions of equations over groups.
\newblock {\em Bull. Amer. Math. Soc.}, 68:603--604, 1962.

\bibitem[Lic62]{Lickorish}
W.~B.~R. Lickorish.
\newblock A representation of orientable combinatorial {$3$}-manifolds.
\newblock {\em Ann. of Math. (2)}, 76:531--540, 1962.

\bibitem[LP24]{LiuPiccirillo}
Beibei Liu and Lisa Piccirillo.
\newblock Bounding the {Dehn} surgery number by 10/8, {arXiv} 2405.16904, 2024.

\bibitem[LW17]{LouderWilton:stacking}
Larsen Louder and Henry Wilton.
\newblock Stackings and the {$W$}-cycles conjecture.
\newblock {\em Canad. Math. Bull.}, 60(3):604--612, 2017.

\bibitem[LW24]{LouderWilton:coherence}
Larsen Louder and Henry Wilton.
\newblock Uniform negative immersions and the coherence of one-relator groups.
\newblock {\em Invent. Math.}, 236(2):673--712, 2024.

\bibitem[Mar24]{Marchand}
Alexis Marchand.
\newblock From letter-quasimorphisms to angle structures and spectral gaps for
  scl, {arXiv} 2402.13856, 2024.

\bibitem[Mil21]{Millard}
Benjamin~James Millard.
\newblock {\em Stackings and One-Relator Products of Groups}.
\newblock Doctoral thesis (ph.d.), University College London, 2021.
\newblock Open access version available via UCL Discovery.

\bibitem[MOT12]{MonodOzawaThom}
Nicolas Monod, Narutaka Ozawa, and Andreas Thom.
\newblock Is an irng singly generated as an ideal?
\newblock {\em Internat. J. Algebra Comput.}, 22(4):1250036, 6, 2012.

\bibitem[OT13]{OsinThom}
Denis Osin and Andreas Thom.
\newblock Normal generation and {$\ell^2$}-{B}etti numbers of groups.
\newblock {\em Math. Ann.}, 355(4):1331--1347, 2013.

\bibitem[Pes08]{Pestov}
Vladimir~G. Pestov.
\newblock Hyperlinear and sofic groups: a brief guide.
\newblock {\em Bull. Symbolic Logic}, 14(4):449--480, 2008.

\bibitem[PSEWS25]{Puder}
Doron Puder, Yotam Shomroni, Danielle Ernst-West, and Matan Seidel.
\newblock Stable invariants of words from random matrices, {arXiv} 2311.17733,
  2025.

\bibitem[RT19]{RivasTriestino}
Crist\'{o}bal Rivas and Michele Triestino.
\newblock One-dimensional actions of {H}igman's group.
\newblock {\em Discrete Anal.}, pages Paper No. 20, 15, 2019.

\bibitem[Sho83]{Short}
Hamish~Buchanan Short.
\newblock {\em Topological methods in group theory: the adjunction problem}.
\newblock PhD thesis, University of Warwick, 1983.

\bibitem[SZ22]{surg4}
Steven Sivek and Raphael Zentner.
\newblock Surgery obstructions and character varieties.
\newblock {\em Trans. Amer. Math. Soc.}, 375(5):3351--3380, 2022.

\bibitem[Thu78]{Thurstonnotes}
William Thurston.
\newblock The geometry and topology of 3-manifolds.
\newblock {\em Lecture note}, 1978.

\bibitem[Wal60]{Wallace}
Andrew~H. Wallace.
\newblock Modifications and cobounding manifolds.
\newblock {\em Canadian J. Math.}, 12:503--528, 1960.

\bibitem[Wil18]{Wilton:surfsubgroup}
Henry Wilton.
\newblock Essential surfaces in graph pairs.
\newblock {\em J. Amer. Math. Soc.}, 31(4):893--919, 2018.

\bibitem[Wil24]{Wilton:curvinv}
Henry Wilton.
\newblock Rational curvature invariants for 2-complexes.
\newblock {\em Proc. A.}, 480(2296):Paper No. 20240025, 39, 2024.

\end{thebibliography}

\end{document}